\documentclass[reqno]{amsart}

\usepackage[utf8]{inputenc}
\usepackage[T1]{fontenc}
\usepackage{fullpage}
\usepackage{color}
\usepackage{anyfontsize}

\usepackage{newtxtext}
\usepackage{mathrsfs}
\usepackage{stmaryrd}
\usepackage{amssymb}
\usepackage{amsfonts}
\usepackage{bbm}
\usepackage{spectralsequences}

\usepackage{mathtools, amsmath, amsthm, amsopn}
\usepackage{thm-restate}
\usepackage{enumitem}
\usepackage[new]{old-arrows}

\usepackage{tikz-cd}
\usetikzlibrary{cd}
\tikzset{
  symbol/.style={
    draw=none,
    every to/.append style={
      edge node={node [sloped, allow upside down, auto=false]{$#1$}}}
  }
}
\usepackage[all]{xy}

\usepackage[
  backend=biber,
  style=alphabetic,
  sorting=nyt,
  giveninits=true,
  maxcitenames=4,
  maxbibnames=4,
  maxnames=4,
  maxalphanames=4
]{biblatex}
\DeclareFieldFormat[misc]{title}{\mkbibemph{#1}}
\DeclareFieldFormat[article]{title}{\mkbibemph{#1}}
\DeclareFieldFormat[article]{journaltitle}{#1}
\renewbibmacro{in:}{}

\usepackage{csquotes}
\usepackage{xcolor}
\definecolor{darkgreen}{rgb}{0,0.8,0}
\definecolor{darkblue}{rgb}{0,0,0.8}
\usepackage{hyperref}
\hypersetup{
  colorlinks=true,
  linktoc=all,
  linkcolor=darkblue,
  citecolor=darkgreen
}
\usepackage{cleveref}
\crefname{lem}{Lemma}{Lemmas}
\Crefname{lem}{Lemma}{Lemmas}
\usepackage{footmisc}

\usepackage[colorinlistoftodos]{todonotes}

\declaretheorem[name=Theorem,numberwithin=section]{thm}

\newtheorem{maintheorem}{Theorem}

\newtheorem{maincor}{Corollary}
\newtheorem{mainprop}[maincor]{Proposition}

\theoremstyle{plain}
\newtheorem{lem}[thm]{Lemma}
\newtheorem{prop}[thm]{Proposition}

\newtheorem{cor}[thm]{Corollary}

\newtheorem*{thm*}{Theorem}

\newtheorem{reptheorem}{Theorem}

\newtheorem{repcorollary}{Corollary}

\theoremstyle{definition}

\newtheorem{df}[thm]{Definition}
\newtheorem{exmp}[thm]{Example}

\theoremstyle{remark}
\newtheorem{rem}[thm]{Remark}

\DeclareMathOperator{\spec}{Spec}

\DeclareMathOperator{\pr}{pr}

\DeclareMathOperator{\id}{id}

\DeclareMathOperator{\ev}{ev}

\newcommand{\C}{\mathbb{C}}
\newcommand{\Cx}{\C^\times}
\newcommand{\Cxd}{\C^\times_{\dd}}
\newcommand{\Cxh}{\C^\times_\hbar}

\newcommand{\R}{\mathbb{R}}
\newcommand{\Z}{\mathbb{Z}}
\newcommand{\Q}{\mathbb{Q}}
\newcommand{\PP}{\mathbb{P}}

\newcommand{\stab}{\mathrm{Stab}}
\newcommand{\stabp}{\mathrm{Stab}^+}
\newcommand{\stabm}{\mathrm{Stab}^-}
\newcommand{\stabpm}{\mathrm{Stab}^{\pm}}
\newcommand{\pt}{\mathrm{pt}}

\newcommand{\Gr}{\mathrm{Gr}}
\newcommand{\Fl}{\mathrm{Fl}}
\newcommand{\oh}{\mathcal{O}}
\newcommand{\ok}{\mathcal{K}}
\newcommand{\lt}{\mathfrak{t}}
\newcommand{\g}{\mathfrak{g}}

\newcommand{\f}{\mathfrak{f}}

\newcommand{\bN}{\mathbf{N}}
\newcommand{\bV}{\mathbf{V}}

\newcommand{\BR}{\mathcal{R}}
\newcommand{\BT}{\mathcal{T}}
\newcommand{\BP}{\mathcal{P}}
\newcommand{\BS}{\mathcal{S}}

\newcommand{\BI}{\mathcal{I}}

\newcommand{\sA}{\mathcal{A}}

\newcommand{\sAs}{\mathcal{A}^{\text{Gr}}}

\newcommand{\cM}{\mathcal{M}}

\newcommand{\cZ}{\mathcal{Z}}

\newcommand{\hG}{\widehat{G}}

\newcommand{\hT}{\widehat{T}}
\newcommand{\hBI}{\widehat{\BI}}

\newcommand{\Waffe}{\widetilde{W}}
\newcommand{\Eff}{\mathrm{Eff}}
\newcommand{\End}{\mathrm{End}}
\newcommand{\loc}{\mathrm{loc}}
\newcommand{\vir}{\mathrm{vir}}
\newcommand{\bbS}{\mathbb{S}}
\newcommand{\cct}{(\!(t)\!)}
\newcommand{\can}{\mathrm{can}}
\newcommand{\PD}{\mathrm{PD}}
\newcommand{\tw}{\mathrm{tw}}

\newcommand{\sAFl}{\mathcal{A}^{\Fl}}
\newcommand{\sAGr}{\mathcal{A}^{\Gr}}
\newcommand{\twistdd}{\Theta_{\dd}}

\newcommand{\xx}{x}
\newcommand{\yy}{y}

\newcommand{\dd}{k}
\newcommand{\hh}{\hbar}
\newcommand{\HH}{\mathcal{H}}
\newcommand{\NH}{\mathcal{NH}}
\newcommand{\SH}{\mathcal{SH}}
\newcommand{\DD}{A}

\newcommand{\bb}{\mathfrak{b}}
\newcommand{\Lie}{\operatorname{Lie}}
\newcommand{\ee}{\mathbf{e}}
\newcommand{\RR}{\Q[\lt][\hbar,\dd]}
\def\TT{\widehat{T}}

\newcommand{\qwa}{\bullet} 

\newcommand{\qw}{\mathbb{W}}

\newcommand{\GG}{\widehat{G}}

\newcommand{\PPP}{\mathcal{P}}
\newcommand{\BBB}{\mathcal{B}}
\newcommand{\supp}{\operatorname{Supp}}

\DeclareFontFamily{U}{mathx}{}
\DeclareFontShape{U}{mathx}{m}{n}{<-> mathx10}{}
\DeclareSymbolFont{mathx}{U}{mathx}{m}{n}
\DeclareMathAccent{\widehat}{0}{mathx}{"70}
\DeclareMathAccent{\widecheck}{0}{mathx}{"71}

\begin{document}
\title{Iwahori--Coulomb Branches, Stable Envelopes, and
Quantum Cohomology of Cotangent Bundles of Flag Varieties}
\author{Ki Fung Chan, Kwokwai Chan, Chi Hong Chow, Chin Hang Eddie Lam}

\address{\parbox{\linewidth}{The Institute of Mathematical Sciences and Department of Mathematics,\\ The Chinese University of Hong Kong, Shatin, Hong Kong\\ \vspace{-.22cm}}}
\email{kfchan@math.cuhk.edu.hk}

\address{Department of Mathematics, The Chinese University of Hong Kong, Shatin, Hong Kong}
\email{kwchan@math.cuhk.edu.hk}

\address{\parbox{\linewidth}{ Max Planck Institute for Mathematics, 53111 Bonn, Germany\\ Department of Mathematics, Virginia Tech, Blacksburg, VA 24061, USA\\ \vspace{-.22cm}}}
\email{chow@vt.edu}

\address{Department of Mathematics, The Chinese University of Hong Kong, Shatin, Hong Kong}
\email{echlam@math.cuhk.edu.hk}

\begin{abstract}
    We consider Iwahori--Coulomb branches $\mathcal{A}_{G,\mathbf{N},\mathbf{V}}^{\mathrm{Fl}}$, which are the affine flag analogs of the original Coulomb branches $\mathcal{A}_{G,\mathbf{N}}^{\mathrm{Gr}}$ defined by Braverman, Finkelberg, and Nakajima. For any conical symplectic resolution $X$, we prove that the $\mathcal{A}_{G,\mathbf{N},\mathbf{V}}^{\mathrm{Fl}}$-action on the localized equivariant quantum cohomology of $X$, induced by shift operators, satisfies a polynomiality property in terms of stable envelopes.
    
    We then study the case $X = T^*(G/P)$, the cotangent bundle of a flag variety, for which the Iwahori--Coulomb branch is isomorphic to the trigonometric double affine Hecke algebra $\mathcal{H}_{G,\hbar,k}$. The polynomiality property enables us to compute explicitly the above action in terms of the Demazure--Lusztig elements and stable envelopes. Applications include:
    \begin{enumerate}
    \item Computation of the Iwarhori--Coulomb branch action for $G/P$ by taking the confluent limit, recovering Peterson--Lam--Shimozono's theorem.
    
    \item Construction of an explicit Namikawa--Weyl group action on the equivariant quantum cohomology of $T^*(G/P)$ that preserves the quantum product, extending a result of Li--Su--Xiong.
    
    \item Proof of a conjecture of Braverman--Finkelberg--Nakajima stating that, up to a shift of the dilation parameter, $\mathcal{A}_{G,\mathfrak{g}^*}^{\mathrm{Gr}}$ is isomorphic to the spherical subalgebra of $\mathcal{H}_{G,\hbar,k}$.
    \end{enumerate}
\end{abstract}
\maketitle

\section{Introduction}
Let $G$ (called the \emph{gauge group}) be a reductive group, $F=(\Cx)^s$ (called the \emph{flavor-symmetry group}) with Lie algebra $\f$, and $\bN$ a $\GG=G\times F$-representation. In \cite{shiftoperatorsCoulombbranches}, three of the present authors constructed, for any smooth semiprojective $\GG$-variety $X$ that is $\GG$-equivariantly proper over $\bN$, an action
\begin{equation}\label{intro_BFN_action_eqn}
    \sAs_{G,\bN} \otimes_{\Q[\f][\hbar]} QH^{\bullet}_{\GG\times\Cxh}(X) \to QH^{\bullet}_{\GG\times\Cxh}(X)
\end{equation}
of the \emph{flavor-deformed} Coulomb branch algebra $\sAs_{G,\bN}$ associated with $\bN$ (\cite{BFN}, \cite{2drole}) on the equivariant quantum cohomology of $X$. More precisely, let $\Gr_G$ be the affine Grassmannian of $G$, $T\subseteq G$ a maximal torus (with Lie algebra $\lt$), $\TT:=T\times F$ (with Lie algebra $\hat\lt=\lt\oplus\f$) and $\Cxh$ the group of loop rotation. Then $\sAs_{G,\bN}$ is a subalgebra of $H^{\smash{\GG\times \mathbb{C}^{\times}_{\hbar}}}_{\bullet}(\Gr_G)$ equipped with the convolution product, and \eqref{intro_BFN_action_eqn} is the restriction of
\begin{equation}\label{intro_BFN_action_eqn2}
    \bbS^\Gr_X:H^{\TT\times \Cxh}_{\bullet}(\Gr_G) \otimes_{\Q[\f][\hbar]} QH^{\bullet}_{\TT\times\Cxh}(X)_\loc \to QH^{\bullet}_{\TT\times\Cxh}(X)_\loc
\end{equation}
induced by the shift operators \cite{OP, BMO, Iritani,MO}, which themselves are defined using Gromov--Witten invariants of certain $X$-bundles over $\PP^1$, called the Seidel spaces \cite{Seidel}. A key insight is that Coulomb branches are intrinsically related to the properness of the evaluation morphisms on the moduli spaces of stable maps to families of Seidel spaces, which implies that \eqref{intro_BFN_action_eqn} is defined without localization.

Consider the affine flag analog $\sAFl_{G,\bN,\bV}$ of $\sAs_{G,\bN}$, called the (flavor-deformed) \emph{Iwahori--Coulomb branch algebra} (see e.g., \cite{Webster,GKO}), which depends on an additional $B^-$-invariant subspace $\bV \subseteq \bN$ and is realized as a subalgebra of $H^{\smash{\TT\times \Cxh}}_{\bullet}(\Fl_G)$; here, $B^-$ is a Borel subgroup containing $T$, and $\Fl_G$ is the affine flag variety of $G$. Our first main result is a parallel construction for $\sAFl_{G,\bN,\bV}$. The resulting action is no longer defined without localization, but it satisfies a polynomiality property.
\begin{maintheorem}\label{Theorem A}
Let $X$ be a smooth semiprojective $\GG$-variety with $X^{\hT}$ compact, and $f \colon X \to \bN$ a $\GG$-equivariant proper morphism. Then the shift operators induce a graded $\Q[\f][\hbar]$-bilinear map
\begin{equation*}
    \bbS_X^{\Fl}:H^{\TT\times \Cxh}_{\bullet}(\Fl_G) \otimes_{\Q[\f][\hbar]} QH^\bullet_{\TT\times\Cxh}(X)_{\loc} \to QH^\bullet_{\TT\times\Cxh}(X)_{\loc}
\end{equation*}
defining an $H^{{\TT\times \Cxh}}_{\bullet}(\Fl_G)$-module structure on $QH_{{\TT\times\Cxh}}^\bullet(X)_\loc$. Furthermore, for any $\Gamma \in \sAFl_{G,\bN,\bV}$ and $\gamma, \gamma'\in H^\bullet_{{\TT\times\Cxh}}(X)$ such that $\supp(\gamma) \subseteq f^{-1}(\bV)$ and $\supp(\gamma') \cap f^{-1}(\bV) \subseteq f^{-1}(0)$, we have
\[(\bbS_X^{\Fl}(\Gamma,\gamma), \gamma')^{\TT}_{X} \in \Q[\hat\lt][\hbar][[q_G]],\]
 where $(\cdot, \cdot)^{\TT}_X$ is the $\TT$-equivariant Poincar\'e pairing of $X$. 

Here, we identify $\Q[\hat\lt]$ with $H_{\TT}^\bullet(\pt)$, $\Q[[q_G]]$ is the ring of equivariant Novikov variables associated with $X$, and $\Q[\hat\lt][\hbar][[q_G]] := \bigl(\Q[\hat\lt][\hbar] \bigr)\widehat{\otimes}_\Q \Q[[q_G]]$. The subscript $\loc$ denotes the localization with respect to the set of nonzero homogeneous elements in $H_{\TT\times\Cxh}^\bullet(\pt)$.
\end{maintheorem}

\Cref{Theorem A} may be viewed as an algebro-geometric generalization of \cite[Theorem~1.1]{GMP}, where an Iwahori--Coulomb branch action for the case $\bN=\bV=\mathbf{0}$ (which forces $X$ to be compact) was constructed.

A rich source of examples for \Cref{Theorem A} is provided by conical symplectic resolutions equipped with a Hamiltonian $G$-action. Throughout the remainder of this introduction we work in this setting. We take $F=\Cxd$ acting via the conical action. Such resolutions admit a proper morphism $f \colon X \to \bN$ to a $\GG$-representation. Let $\tau$ be a generic cocharacter of $T$ in the dominant Weyl chamber. Let $\bV \subseteq \bN$ be the negative weight space of $\tau$. Consider the \emph{stable envelope} \cite{MO}
\begin{equation}\label{intro_stable_envelop}
    \stabpm \colon H^\bullet_{\TT}(X^T) \to H^\bullet_{\TT}(X)
\end{equation}
associated with the chambers containing $\pm\tau$. By the defining properties of stable envelopes, we obtain the following corollary of \Cref{Theorem A}:

\begin{maincor}\label{thm_BFN_action_symplectic_resolution}
    The shift operators induce a graded algebra action $\bbS_X^{\Fl}$ of $H^{\TT\times\Cxh}_\bullet(\Fl_G)$ on $QH^\bullet_{\TT\times\Cxh}(X)_{\loc}$ such that for any $\Gamma \in \sAFl_{G,\bN,\bV}$ and $\gamma, \gamma' \in H^\bullet_{\TT}(X^{T})$, we have
    \begin{equation*} 
        (\bbS_X^{\Fl}(\Gamma, \stabm(\gamma)), \stabp(\gamma'))^{\TT}_{X} \in \Q[\mathfrak{t}][\hbar, \dd][[q_G]],
    \end{equation*}
    that is, the pairing $(\bbS_X^{\Fl}(\Gamma, \stabm(\gamma)), \stabp(\gamma'))^{\TT}_{X}$ is defined without localization of the equivariant parameters. 
\end{maincor}

In the second part of the paper, we specialize the action from \Cref{thm_BFN_action_symplectic_resolution} to the case $X = T^*\PPP$, the cotangent bundle of the flag variety $\PPP := G/P$ for an arbitrary parabolic subgroup $P \subseteq G$. In this case, we can take $(\bN,\bV)=(\g^*,(\mathfrak{b}^-)^{\perp})$. We let the flavor symmetry group $F=\Cxd$ act by dilation on the cotangent fibers, as well as on $\g^*$. We first construct a basis of the corresponding Iwahori--Coulomb branch algebra that will be useful.
\begin{mainprop}\label{thm_BFN_basis}
    The Iwahori--Coulomb branch algebra $\sAFl_{G, \mathfrak{g}^*, (\mathfrak{b}^-)^{\perp}}$ admits a $\Q[\lt][\hbar, \dd]$-basis $\{\DD_{\xx}\}_{\xx \in \Waffe}$ indexed by the extended affine Weyl group $\Waffe \simeq (\Fl_G)^T$. Each $\DD_{\xx}$ is uniquely determined by the following properties:
    \begin{enumerate}
        \item $\DD_{\xx}$ is homogeneous of degree $0$; and
        \item $\DD_{\xx} \equiv [\xx] \pmod{\dd}$.
    \end{enumerate}
    Moreover, we have
    \begin{equation}\label{eq_BFN_basis}
        \DD_\xx = (-\dd)^{\ell(\xx)} [C_{\leq\xx}] + \text{terms with smaller $\dd$-degree},
    \end{equation}
    where $C_{\leq\xx}\subseteq\Fl_G$ is the Iwahori orbit closure associated with the $T$-fixed point $\xx$.
\end{mainprop}
The proof actually shows that the elements $\DD_{\xx}$ correspond to the \emph{Demazure--Lusztig elements}, i.e., the elements which act as Demazure--Lusztig operators under the difference-rational polynomial representation (see Section \ref{subsection_construction_of_a_basis} for details). 

As a corollary, we recover the following result from \cite[Theorem 6.1]{GKO} (cf.\ \cite{VV} for a related construction of the double affine Hecke algebra using $K$-theory).

\begin{maincor}\label{cor_Coulomb_branch=tDAHA}
    The Iwahori--Coulomb branch algebra $\sAFl_{G, \mathfrak{g}^*, (\mathfrak{b}^-)^{\perp}}$ is isomorphic to the trigonometric double affine Hecke algebra (tDAHA) $\HH_{G,\hbar,k}$ associated with $G$.
\end{maincor}

The central result of this part is an explicit description of the Iwahori--Coulomb branch algebra action in terms of the stable envelope $\stabm$:

\begin{maintheorem}\label{Theorem B}
    The $\sAFl_{G, \mathfrak{g}^*, (\mathfrak{b}^-)^{\perp}}$-action on $QH^\bullet_{\TT\times\Cxh}(T^*\PPP)_{\loc}$ from \Cref{Theorem A} satisfies
    \begin{equation}\label{thm_compute_action_eq}
        \bbS_{T^*\PPP}^{\Fl}(\DD_{\xx}, \stabm(u)) = (-1)^{d_{u,\lambda}}\, q^{u^{-1}(\lambda)}\, \stabm(wu)
    \end{equation}
    for any $\xx = w t_\lambda \in \Waffe$ and $u\in W/W_P$. Here, $\stabm(u)$ denotes the image of $1\in H_{\TT}^{\bullet}(uP)\subseteq H_{\TT}^{\bullet}((T^*\PPP)^T)$ under the stable envelope $\stabm$.
\end{maintheorem}
For an element $\xx$ in the (finite) Weyl group $W$, Formula \eqref{thm_compute_action_eq} was proved in \cite[Theorem 4.3]{MNS}. (After homogenization, their operators $\mathcal{T}_i^L$ correspond to our $\DD_{s_i}$, and their CSM classes coincide with our stable basis.) See also \cite{Postnikov, Belkale, CMP} for related discussions.

Let us explain the idea of the proof of \Cref{Theorem B}. It is equivalent to proving
\begin{equation*}
    (\bbS_{T^*\PPP}^{\Fl}(\DD_{\xx}, \stabm(u)), (-1)^{\dim\PPP}\stabp(v))^{\TT}_{T^*\PPP} = (-1)^{d_{u,\lambda}} q^{u^{-1}(\lambda)} \delta_{wu, v}.
\end{equation*}
The left-hand side lies in $\Q[\lt][\hbar, \dd][[q_G]]$ by \Cref{thm_BFN_action_symplectic_resolution}, and its cohomological degree can be checked to be $0$ using \Cref{thm_BFN_basis}. Hence, it must lie in $\Q[[q_G]]$ and can then be computed by specializing $\dd=0$ and using the fact that nontrivial non-equivariant Gromov--Witten invariants of a symplectic variety vanish.

The significance of \Cref{Theorem B} is that the algebra action \eqref{thm_compute_action_eq} arises from shift operators, giving a geometric construction of the tDAHA action on $QH_{\TT\times\Cxh}^\bullet(T^*\PPP)_\loc$. This leads to the following applications.

\subsubsection*{Confluent limit}
Consider the case $\bN=\bV=\mathbf{0}$ with trivial flavor symmetry group. Then the Iwahori--Coulomb branch is isomorphic to the \emph{affine nil-Hecke algebra} $\NH_G\simeq H_\bullet^{T\times\Cxh}(\Fl_G)$.
It is known (e.g., \cite{BMO}) that the Gromov--Witten invariants of $\PPP$ can be recovered from those of $T^*\PPP$ by taking the limit $\dd\to\infty$. This enables us to compute the Iwahori--Coulomb branch action for $\PPP$ using \Cref{Theorem B}:

\begin{maintheorem}\label{thm_compute_action_G/P}
    The $\NH_G$-action on $QH_{T\times\Cxh}^\bullet(\PPP)$ from \Cref{Theorem A} satisfies
    \begin{equation*}
        \bbS_{\PPP}^{\Fl}([C_{\leq\xx}], \sigma(u)) = 
        \begin{cases} 
            q^{u^{-1}(\lambda)}\, \sigma(wu) & \text{if $(\xx, u)$ is $P$-allowed}, \\
            0 & \text{otherwise},
        \end{cases}
    \end{equation*}
    for any $\xx = w t_\lambda \in \Waffe$ and $u \in W/W_P$. Here, $\sigma(u):=\PD \left(\left[\overline{B^-uP/P}\right]\right)$ is the Schubert class associated with $u$, and $C_{\leq\xx}\subseteq\Fl_G$ is the Iwahori orbit closure associated with $\xx$.
\end{maintheorem}
The definition of the $P$-allowed condition is given in \Cref{section_compute_action_G/P}. \Cref{thm_compute_action_G/P} implies Peterson's celebrated ``Quantum Equals Affine'' \cite{peterson,Lam-Shimozono,Leung-Li,Chow,shiftoperatorsCoulombbranches}.

\begin{maincor}\label{thm_new_proof_of_PLS}
There is a graded ring homomorphism $\Upsilon_P:H^T_\bullet(\Gr_G)\to H_T^\bullet(\PPP)[q_G]$ satisfying
\begin{equation*}
    \Upsilon_P([C_{\leq\lambda}]) =
\begin{cases}
    q^{\lambda^-}\sigma(w_\lambda^-)
        & \text{if $\lambda$ is $P$-allowed},\\
    0 & \text{otherwise.}
\end{cases}
\end{equation*}
When $P=B$, it becomes an isomorphism after localizing $H^T_\bullet(\Gr_G)$ by the classes $[C_{\leq \lambda}]$ for all dominant $\lambda$.
\end{maincor}

\subsubsection*{Namikawa--Weyl group action}
We define an action $\qwa$ of $\qw_P:=N_W(W_P)/W_P$ on $H_{\TT}^\bullet(T^*\PPP)[q_G]$ satisfying
\begin{equation*}
    w\qwa\bigl(q^{\lambda}\stabm(u)\bigr)
    = (-1)^{\ell_P(w)}\, q^{w(\lambda)}\, \stabm(uw^{-1})
\end{equation*}
and preserves the submodule $H_{\hG}^\bullet(T^*\PPP)$. The name of this action comes from the fact that the \emph{Namikawa--Weyl group} \cite{NamikawaII, NamiNilpII} of $T^*\PPP$ is a subgroup of $\qw_P$, and the two are equal when $G$ is of type~A \cite{NamikawaII}. 

We apply \Cref{Theorem B} to describe the Coulomb branch algebra action on $QH_{\GG\times\Cxh}^{\bullet}(T^*\PPP)$ originally constructed in \cite{shiftoperatorsCoulombbranches} in terms of this group action.

\begin{mainprop}\label{thm_image_shift=qwinvariants}
The $\Q[\lt][\hbar, \dd]$-linear map 
\[
\begin{array}{ccccc}
  \Psi_{T^*\PPP}&:&H_{\bullet}^{\TT\times\Cxh}(\Gr_G) &\longrightarrow &QH_{\TT\times\Cxh}^\bullet(T^*\PPP)_{\loc}\\ [.9em]
&&a&\longmapsto&\bbS_{T^*\PPP}^{\Gr}(a,1)
\end{array}
\]
(see \eqref{intro_BFN_action_eqn2}) satisfies
\begin{enumerate}
    \item $\Psi_{T^*\PPP}(\sAGr_{G,\g^*})$ is contained in $\bigl(H_{\GG\times\Cxh}^\bullet(T^*\PPP)[q_G]\bigr)^{\qw_P}$.
    \item The $\Q(q_G)$-span of $\Psi_{T^*\PPP}(\sAGr_{G,\g^*})$ contains $H_{\GG\times\Cxh}^{\bullet}(T^*\PPP)[q_G]$.
    \item $\Psi_{T^*\BBB}$ restricts to a $\Q[\lt]^W[\hbar, \dd]$-module isomorphism 
    \[
    \sAGr_{G,\g^*}\simeq \bigl(H_{\GG\times\Cxh}^\bullet(T^*\BBB)[q_G]\bigr)^{\qw_B}.
    \]
\end{enumerate}
\end{mainprop}

By a property of shift operators, the specialization of $\Psi_{T^*\PPP}$ at $\hbar = 0$ is a ring homomorphism. Combined with \Cref{thm_image_shift=qwinvariants}, this yields the following two corollaries.

\begin{maincor}\label{thm_LSX}
    The submodule $H_{\GG}^{\bullet}(T^*\PPP)(q_G)$ is closed under the quantum product $\star$. Furthermore, $\star$ is preserved by the $\qw_P$-action, that is,
    \begin{equation*}
        w\qwa(\gamma_1\star \gamma_2)= (w\qwa \gamma_1)\star(w\qwa \gamma_2)
    \end{equation*}
    for all $w\in \qw_P$ and $\gamma_1,\gamma_2\in H_{\GG}^{\bullet}(T^*\PPP)(q_G)$.
\end{maincor}

\begin{maincor}\label{cor_image_psi}
    There is a graded ring isomorphism
    \begin{equation*}
        \sAGr_{G,\g^*}|_{\hbar=0} \simeq \bigl(\bigl(H_{\GG}^{\bullet}(T^*\BBB)[q_G]\bigr)^{\qw_B},\star \bigr).
    \end{equation*}
\end{maincor}
For $P=B$, Corollary \ref{thm_LSX} has been proved in \cite{LSX}.

\subsubsection*{Spherical subalgebra}
The \emph{spherical subalgebra} $\SH_{G,\hbar,\dd}$ of the trigonometric double affine Hecke algebra $\HH_{G,\hbar,k}$ is defined to be $\ee'\HH_{G,\hbar,k}\ee'$, where $\ee'$ is the idempotent element corresponding to $\ee:=\frac{1}{|W|}\sum_{w\in W}\DD_w\in \sAFl_{G,\g^*,(\bb^-)^{\perp}}$ under the isomorphism from \Cref{cor_Coulomb_branch=tDAHA}. As an application of \Cref{thm_image_shift=qwinvariants}, we provide a proof of the following result, which was conjectured by Braverman, Finkelberg and Nakajima \cite[3(x)(b)]{BFN} and proved in \cite{KN} for $G=\operatorname{\textbf{GL}}_n$. 

\begin{maintheorem}\label{thm_spherical_subalgebra}
   The (flavor-deformed) Coulomb branch algebra $\sAGr_{G,\g^*}$ associated with the coadjoint matter is isomorphic to the spherical subalgebra $\SH_{G,\hh,\dd-\hh}$ of the trigonometric double affine Hecke algebra with parameter shift $\dd\mapsto \dd-\hh$. 
\end{maintheorem}

The above two rings cannot be (graded) isomorphic without the parameter shift $k\mapsto k-\hbar$, as illustrated by \Cref{counterexample} below.\footnote{This contradicts \cite[Corollary 6.3]{GKO} as stated. Thus, we believe their result is incorrect.} We remark that \cite{KN} uses the same convention for loop rotation as \cite[Convention, Section~2]{BFN}, and also adopts the change of variables introduced in \cite[A(iii)]{BFNslice}, so their Coulomb branch differs from ours by $\dd\to\dd+\hh$. Therefore, \Cref{thm_spherical_subalgebra} agrees with \cite[Theorem 1.1]{KN} when $G=\operatorname{\textbf{GL}}_n$.

Let us sketch the proof. Let $\twistdd$ be the unique ring automorphism of
$H_{\bullet}^{\TT\times\Cxh}(\Fl_G)\simeq H_{\bullet}^{T\times\Cxh}(\Fl_G)[\dd]$ sending
$\dd$ to $\dd-\hbar$ and fixing the subring $H_{\bullet}^{T\times\Cxh}(\Fl_G)$. By identifying $\HH_{G,\hbar,k}$ with $\sAFl_{G,\g^*,(\bb^-)^{\perp}}$, \Cref{thm_spherical_subalgebra} reduces to exhibiting a ring isomorphism
\[ \sAGr_{G,\g^*}\simeq \twistdd\big(\ee\sAFl_{G,\g^*,(\bb^-)^{\perp}}\ee\big).\]
The projection $\pi:\Fl_G\to \Gr_G$ induces the pushforward map $\pi_*:H_\bullet^{\TT\times\Cxh}(\Fl_G)\to H_\bullet^{\TT\times\Cxh}(\Gr_G)$. It can be checked directly that it restricts to a ring homomorphism 
\begin{equation*}
    \psi:\twistdd\big(\ee\sAFl_{G,\g^*,(\bb^-)^{\perp}}\ee\big)\to H_{\bullet}^{\GG\times\Cxh}(\Gr_G)
\end{equation*}
which intertwines any (Iwahori--)Coulomb branch actions $\bbS_X^{\Fl}$ and $\bbS_X^{\Gr}$. Our goal is to show that $\psi$ is bijective onto $\sAGr_{G,\g^*}$. Take $X=T^*\BBB$. By \Cref{thm_image_shift=qwinvariants}(3), it suffices to prove that the image of $\twistdd\big(\ee\sAFl_{G,\g^*,(\bb^-)^{\perp}}\ee\big)$ under the map $a\mapsto \bbS_{T^*\BBB}^{\Fl}(a,1)$ coincides with$\bigl(H_{\GG\times\Cxh}^\bullet(T^*\BBB)[q_G]\bigr)^{\qw_B}$.

The main challenge is to show that the image lies in the nonlocalized quantum cohomology, due to the presence of the twist map $\twistdd$ and the lack of an explicit formula for computing $\bbS_{T^*\BBB}^{\Fl}(-,1)$. We overcome it by analyzing the shift operator $\bbS_{t_{\dd}}$ associated with the cocharacter $t_\dd$ corresponding to the inclusion $\Cxd\subseteq\TT$. We prove that 
\[ \bbS_{T^*\BBB}^{\Fl}(\twistdd(a),1) = \bbS_{t_{\dd}}^{-1}\circ \bbS_{T^*\BBB}^{\Fl}(a,\bbS_{t_{\dd}}(1))\]
for any $a$, and 
\[ \bbS_{t_{\dd}}(1)=f e(T^*\BBB)=\pm f\sum_{u\in W}\stabm(u)\]
for some power series $f$ in the equivariant Novikov variables. By \Cref{Theorem B}, this gives 
\[ \bbS_{T^*\BBB}^{\Fl}\left(\twistdd\big(\ee\sAFl_{G,\g^*,(\bb^-)^{\perp}}\ee\big),1\right)\subseteq f \bbS_{t_{\dd}}^{-1}\left(e(T^*\BBB)\cup H_{\GG\times\Cxh}^{\bullet}(T^*\BBB)[q_G]\right) .\]
The proof is then completed by a further analysis of $\bbS_{t_{\dd}}$.

\begin{exmp}\label{counterexample}    
Let $G:= \operatorname{\textbf{PGL}}_2$ and $T\simeq\Cx$ be its standard maximal torus. Let $\alpha$ be the unique positive root. Write $a:=\alpha$ regarded as an element of $H^2_T(\pt)$, $u:=[t_{\alpha^{\vee}/2}]$, and $\xx:=s_{\alpha}t_{-\alpha^{\vee}/2}$. Note that $\xx$ is the unique length 0 element of $\Waffe$ different from $e$. The ring $\SH_{G,\hh,\dd}$ is generated by $\Q[\hh,\dd,a^2]$, $P:=2\ee\DD_{\xx}\ee$ and $Q:=-2\ee a\DD_{\xx}\ee$. Consider the pushforward map $\pi_*$ associated with the projection $\pi:\Fl_G\to\Gr_G$. We have
    \[ \pi_*(P) = \left(\frac{\dd-a}{-a}\right)u+\left(\frac{\dd+a}{a}\right)u^{-1}\quad\text{ and }\quad \pi_*(Q) = (\dd-a)u+(\dd+a)u^{-1}.\]
 On the other hand, the ring $\sAGr_{G,\g^*}$ is generated by $\Q[\hh,\dd,a^2]$,
\[  P':=\left(\frac{\dd-\hh-a}{-a}\right)u+\left(\frac{\dd-\hh+a}{a}\right)u^{-1}\quad\text{ and }\quad Q':= (\dd-\hh-a)u+(\dd-\hh+a)u^{-1}. \]
This follows from \cite[Proposition~A.2]{BFNslice} with $\lambda=\alpha^\vee/2$. For example, their formula for $\lambda=\alpha^\vee/2,f=1$ simplifies to\footnote{In our conventions, $\Cxh$ does not act on the representation.}
\begin{equation*}
   P'= \frac{e\bigl(t^{-1}\g_{-\alpha}\bigr)}{e\bigl(T_{t_{\alpha^{\vee}/2}}(C)\bigr)}u
    +\frac{e\bigl(t^{-1}\g_{\alpha}\bigr)}{e\bigl(T_{t_{-\alpha^{\vee}/2}}(C)\bigr)}u^{-1},
\end{equation*}
where $C$ is the $G_\oh$-orbit of the point $t_{\alpha^{\vee}/2}\in\Gr_G$. The expression for $Q'$ can be obtained by setting $\lambda=\alpha^\vee/2$ and $f=a$.
This gives $\pi_*(\SH_{G,\hh,\dd-\hh})=\sAGr_{G,\g^*}$, and hence \Cref{thm_spherical_subalgebra} for this case follows.

Note that these two sets do not coincide without the parameter shift $\dd\mapsto\dd-\hh$. In fact, we prove that there is no graded ring isomorphism from $\SH_{G,\hh,\dd}$ onto $\sAGr_{G,\g^*}$ that preserves $\Q[\hh,\dd,a^2]$. Suppose $\phi$ is such an isomorphism. It is straightforward to verify the following relations:
\begin{enumerate}
    \item $[P,a^2]=-2\hh Q+\hh^2P$\quad and \quad $[P',a^2]=-2\hh Q'+\hh^2P'$

    \item $[P,Q]=-\hh P^2+4\hh\ee$\quad and \quad $[P',Q']=-\hh (P')^2+4\hh$

    \item $Q^2=a^2P^2-\hh QP-4a^2\ee+4\dd(\dd+\hh)\ee$ \quad and \quad  $(Q')^2=a^2(P')^2-\hh Q'P'-4a^2+4\dd(\dd-\hh)$
\end{enumerate}
(Note that $\ee$ is the unit of $\SH_{G,\hh,\dd}$.) Observe that the degree 0 parts of the two rings are $\Q[P]$ and $\Q[P']$ respectively, and hence we must have $\phi(P)=cP'+d$ for some $c,d\in\Q$ with $c\ne 0$. By (1), we obtain $\phi(Q)=cQ'+\frac{d\hh}{2}$, and then by (2), we get $c=\pm 1$ and $d=0$, but this contradicts (3).
\end{exmp}

\begin{rem} 
We can also consider the \emph{antispherical subalgebra} $\ee'_-\mathcal H_{G,\hbar,k}\ee'_-$ of $\mathcal H_{G,\hbar,k}$, where $\ee_-'$ is the idempotent corresponding to $\ee_-:=\frac{1}{|W|}\sum_{w\in W}(-1)^{\ell(w)}\DD_w\in \sAFl_{G,\g^*,(\bb^-)^\perp}$. By a similar strategy, one can show that $\pi_*$ induces an algebra isomorphism 
\begin{equation*}
    \Delta^{-1}\ee'_-\mathcal H_{G,\hbar,k}\ee'_-\Delta\simeq \sAGr_{G,\g^*},
\end{equation*}
where $\Delta := \prod_{\alpha\in R^+} (\dd+\alpha)$. This proves that the spherical and antispherical subalgebras of $\mathcal H_{G,\hbar,k}$ are isomorphic, up to a shift of the dilation parameter. The analogous statement for the rational double affine Hecke algebra is known (see e.g., \cite[Proposition 4.11]{BEG03}).
\end{rem}

\subsection*{Acknowledgements}
We thank Yoshinori Namikawa for answering a question regarding the Namikawa--Weyl group.

K. F. Chan would like to thank the support provided by The Institute of Mathematical Sciences at The Chinese University of Hong Kong. K. Chan and C. H. E. Lam were substantially supported by grants of the Hong Kong Research Grants Council (Project No. CUHK14305023, CUHK14302524 \& CUHK14310425). C. H. Chow gratefully acknowledges the hospitality and financial support of the Max Planck Institute for Mathematics in Bonn, where an early part of this work was carried out.

\subsection*{Notation}
In this paper, $G$ denotes a connected complex reductive group, $F=(\Cx)^s$, $T \subseteq G$ a maximal torus. We write $\hG\coloneq G\times F$ and $\hT\coloneq T\times F$.

$W = N_G(T)/T$ denotes the Weyl group of $G$, with identity element $e$. The coweight lattice of $T$ is denoted by $\Lambda$. We let $\Waffe$ denote the extended affine Weyl group $W\ltimes\Lambda$. Any element of $\Waffe$ is written as $\xx=wt_\lambda$ with $w\in W$ and $\lambda\in \Lambda$.

We write $R$ for the set of roots of $G$. We fix a Borel subgroup $B\supseteq T$, which determines a set of positive roots $R^+\subseteq R$. Let $B^-$ be the opposite Borel. We denote by $\lt$, $\mathfrak{b}$, $\mathfrak{b}^-$ and $\mathfrak{g}$ the Lie algebras of $T$, $B$, $B^-$ and $G$, respectively.

For any graded $\Q[\hat\lt][\hbar]$-module $M$, $M_\loc$ denotes the localization of $M$ with respect to the set $S$ of nonzero homogeneous elements in $\Q[\hat\lt][\hbar]$. Note that $M_\loc$ is a graded $\Q[\hat\lt][\hbar]_\loc$-module.

\section{Part I: Iwahori--Coulomb branch action on torus equivariant quantum cohomology}\label{section_part1}
\subsection{Iwahori--Coulomb branches}\label{subsection_affine_flag_coulomb_branch}
In this subsection, we give a short treatment of Coulomb branches and set up notation that will be important for later sections. Algebraic varieties are always defined over $\C$, and (co)homology groups are taken with coefficients in $\Q$.

Denote $\oh = \C[[t]]$ and $\ok = \C\cct$. We set
\begin{align*}
    G_\oh &= G(\C[[t]]), \\
    G_\ok &= G(\C\cct),
\end{align*}
and
\begin{equation*}
    \BI = \{\, g \in G_\oh \mid g(0) \in B^- \,\}.
\end{equation*}
The \emph{affine Grassmannian} $\Gr_G$ and the \emph{affine flag variety} $\Fl_G$ of $G$ are defined as the fppf quotients
\begin{align*}
    \Gr_G & := G_\ok / G_\oh, \\
    \Fl_G & := G_\ok / \BI.
\end{align*}
The $T$-fixed points of $\Gr_G$ are indexed by the coweight lattice $\Lambda$, where each $\lambda\in \Lambda$ corresponds to $t_\lambda\coloneq \lambda(t)G_\oh$. Similarly, the $T$-fixed points of $\Fl_G$ are indexed by the extended affine Weyl group $\Waffe = W \ltimes \Lambda$; for $x=wt_\lambda\in \Waffe$, the corresponding fixed point is, by abuse of notation, denoted by $x\coloneq\dot w\lambda(t)\BI$, where $\dot w\in N_G(T)$ is a lift of $w\in W$.

For $\lambda \in \Lambda$ and $x \in \Waffe$, define the $\BI$-orbits
\begin{align*}
    C_\lambda &= \BI \cdot t_\lambda \subseteq \Gr_G, \\
    C_x &= \BI \cdot \xx \subseteq \Fl_G,
\end{align*}
and let $C_{\leq \lambda}$ and $C_{\leq x}$ denote their closures in $\Gr_G$ and $\Fl_G$, respectively.

Define a partial order on $\Waffe$ (resp.,\ $\Lambda$) by declaring that $\yy \leq \xx$ (resp.,\ $\mu \leq \lambda$) if and only if $C_\yy \subseteq C_{\leq \xx}$ (resp.,\ $C_\mu \subseteq C_{\leq \lambda}$). In particular,
\begin{align*}
    C_{\leq \lambda} &= \bigsqcup_{\mu \leq \lambda} C_\mu, \\
    C_{\leq x} &= \bigsqcup_{y \leq x} C_y.
\end{align*}

\subsubsection*{Pure gauge Coulomb branches}
Let $\Cxh$ be a one-dimensional torus that acts on $G_\ok$ by $z \cdot g(t) = g(zt)$. This induces actions of $\Cxh$ on $\Fl_G$, $\Gr_G$, $\BI$, $G_\oh$, etc. In particular, we can form the semidirect products $\BI \rtimes \Cxh$ and $G_\oh \rtimes \Cxh$, etc.

Following \cite{MV,BFN}, we introduce the following convolution product $\ast_\Fl$ on $H^{\BI \rtimes \Cxh}_\bullet(\Fl_G)$. Consider the diagram
\begin{equation}\label{convolution1}
\begin{tikzcd}
    \Fl_G \times \Fl_G & \ar[l, "p"'] G_\ok \times \Fl_G \ar[r, "q"] & G_\ok \times_{\BI} \Fl_G \ar[r, "m"] & \Fl_G,
\end{tikzcd}
\end{equation}
where $p$ and $q$ are the natural projections, and $m$ is given by $m([g,g']) = [gg']$. Let $\BI$ act on $\Fl_G$ and $G_\ok \times_{\BI} \Fl_G$ from the left, and let $\BI \times \BI$ act on $G_\ok \times \Fl_G$ by
\[
(g_1, g_2) \cdot (g, [g']) = (g_1 g g_2^{-1}, [g_2 g']).
\]
Then $p$ is $(\BI \times \BI) \rtimes \Cxh$-equivariant and $m$ is $\BI \rtimes \Cxh$-equivariant. Moreover, $q$ induces an isomorphism
\begin{equation*}
    q^* : H^{\BI \rtimes \Cxh}_\bullet(G_\ok \times_\BI \Fl_G)
    \xrightarrow{\sim}
    H^{(\BI \times \BI) \rtimes \Cxh}_\bullet(G_\ok \times \Fl_G).
\end{equation*}

The desired convolution product $\ast$ is defined by
\begin{equation}\label{product1}
    m_* \circ (q^*)^{-1} \circ p^* :
    H^{\BI \rtimes \Cxh}_\bullet(\Fl_G)
    \otimes_{\Q[\hbar]}
    H^{\BI \rtimes \Cxh}_\bullet(\Fl_G)\simeq H^{(\BI\times \BI) \rtimes \Cxh}_\bullet(\Fl_G\times \Fl_G)
    \to
    H^{\BI \rtimes \Cxh}_\bullet(\Fl_G).
\end{equation}

Similarly, there are diagrams
\begin{equation}\label{convolution11}
\begin{tikzcd}
    \Gr_G \times \Gr_G & \ar[l, "p"'] G_\ok \times \Gr_G \ar[r, "q"] &
    G_\ok \times_{\BI} \Gr_G \ar[r, "m"] & \Gr_G,
\end{tikzcd}
\end{equation}
\begin{equation}\label{convolution111}
\begin{tikzcd}
    \Fl_G \times \Gr_G & \ar[l, "p"'] G_\ok \times \Gr_G \ar[r, "q"] &
    G_\ok \times_{\BI} \Gr_G \ar[r, "m"] & \Gr_G,
\end{tikzcd}
\end{equation}
which induce 
\begin{align}
    &H^{\BI \rtimes \Cxh}_\bullet(\Gr_G)\otimes_{\Q[\hbar]}H_{\bullet}^{G_\oh\rtimes\Cxh}(\Gr_G)\xrightarrow{-\ast_{\Gr}-} H_\bullet^{\BI\rtimes\Cxh}(\Gr_G);\\
    &H^{\BI \rtimes \Cxh}_\bullet(\Fl_G)\otimes_{\Q[\hbar]} H^{\BI \rtimes \Cxh}_\bullet(\Gr_G)\xrightarrow{-\cdot_{\Gr}-} H^{\BI\rtimes \Cxh}_\bullet(\Gr_G).\label{FlactsonGr}
\end{align}
These give an algebra structure on $H_{\bullet}^{G_\oh\rtimes\Cxh}(\Gr_G)$, as well as an $H^{\BI \rtimes \Cxh}_\bullet(\Fl_G)$-module structure on $H^{\BI \rtimes \Cxh}_\bullet(\Gr_G)$. When there is no risk of confusion, we write $\ast$ for both $\ast_{\Fl}$ and $\ast_{\Gr}$, and $\cdot$ for $\cdot_{\Gr}$. 
\begin{lem}\label{psi_is_ringhomo}
    For $\Gamma\in H^{\BI \rtimes \Cxh}_\bullet(\Fl_G)$, and $\Gamma'\in H^{G_\oh \rtimes \Cxh}_\bullet(\Gr_G)$, we have 
    \[ 
    \Gamma \cdot_\Gr \Gamma'= (\Gamma\cdot_\Gr 1) \ast_\Gr \Gamma'.
    \]
\end{lem}
\begin{proof}
    By localization and linearity, we may assume that $\Gamma=[t_\lambda w]$ and that
$\Gamma'=\sum_{\mu\in\Lambda} [t_\mu]\ast_{\Gr} f_\mu$ with $f_\mu\in \Q[\lt][\hbar,\dd]$.
Since $\Gamma'$ is $W$-invariant, both sides of the equality compute to
$\sum_{\mu\in\Lambda} [t_{\lambda+\mu}]\ast_{\Gr} f_\mu$.
\end{proof}

\subsubsection*{Coulomb branches with matters}\label{BFNcoul}
Let $\bN$ be a complex representation of $G$, and $\bV \subseteq \bN$ be a $B^-$-invariant subspace\footnote{In \cite{BFN} and \cite{KN}, $\Cxh$ is assumed to act on $\bN$, but in our convention it does not.}. We define the following spaces:
\[
\BT_{G,\bN,\bV} = G_\ok \times_{\BI} (\bV \oplus t \bN_\oh), \qquad
\BT_{G,\bN} = G_\ok \times_{G_\oh} \bN_\oh,
\]
\[
\BR_{G,\bN,\bV} = \{ [g, s] \in \BT_{G,\bN,\bV} : g s \in \bV \oplus t \bN_\oh \}, \qquad
\BR_{G,\bN} = \{ [g, s] \in \BT_{G,\bN} : g s \in \bN_\oh \},
\]
\[
\BS_{G,\bN,\bV} = \BT_{G,\bN,\bV} / \BR_{G,\bN,\bV}, \qquad
\BS_{G,\bN} = \BT_{G,\bN} / \BR_{G,\bN}.
\]
We will sometimes skip the subscripts and write $\BT, \BR, \BS$ when no ambiguity appears. We define
\begin{equation*}
    \sAFl_{G,\bN,\bV} = \hat H^{\BI\rtimes\Cxh}_\bullet(\BR_{G,\bN,\bV}), \qquad
    \sAs_{G,\bN} =\hat  H^{G_\oh\rtimes\Cxh}_\bullet(\BR_{G,\bN}),
\end{equation*}
where $\hat H$ denotes the Borel--Moore homology.

We explain the definition for $\sAFl_{G,\bN,\bV}$, and that for $\sAs_{G,\bN}$ is similar. For each $x \in \Waffe$, we choose an integer $d$ large enough so that $g \cdot t^d \bN \subseteq \bV \oplus t \bN_\oh$ for $g \in C_{\leq x}$, and define
\begin{equation*}
    \BT^d_{\leq x} :=
        G_{\ok}^{\leq x} \times_\BI ((\bV + t \bN_\oh)/t^d \bN_\oh),
    \qquad
    \BR^d_{\leq x} :=
        \{ (g,n) \in \BT^d_{\leq x} :
            g \cdot n \in (\bV + t \bN_\oh)/t^d \bN_\oh \},
\end{equation*}
where $G_{\ok}^{\leq x}$ is the preimage of $C_{\leq x}$ under the map $G_\ok \to \Gr_G$. We use the symbol $z_x^*$ to denote the corresponding Gysin pullbacks
\[
\hat H^{\BI \rtimes \Cxh}_{\bullet}(\BT^d_{\leq x})
\longrightarrow
H^{\BI \rtimes \Cxh}_\bullet(C_{\leq x}).
\]
\begin{df}
The \emph{Iwahori--Coulomb branch} $\sAFl_{G,\bN,\bV}$ is identified with the sum, over $x \in \Waffe$, of the images under the compositions
\[
\hat H^{\BI \rtimes \Cxh}_\bullet(\BR^d_{\leq x})
\longrightarrow
\hat H^{\BI \rtimes \Cxh}_\bullet(\BT^d_{\leq x})
\stackrel{z_x^*}{\longrightarrow}
H^{\BI \rtimes \Cxh}_\bullet(C_{\leq x})
\longrightarrow
H^{\BI \rtimes \Cxh}_\bullet(\Fl_G),
\]
where the first and last maps are pushforward maps along inclusions.   
\end{df}
It was shown in \cite{BFN} that $\sAs_{G,\bN}$ is a subalgebra of $H^{G_\oh \rtimes \Cxh}_\bullet(\Gr_G)$. A similar proof also shows that $\sAFl_{G,\bN,\bV}$ is a subalgebra of $H^{\BI \rtimes \Cxh}_\bullet(\Fl_G)$.

\subsubsection*{An alternative description}
We give an alternative description of the Iwahori--Coulomb branch algebra $\sAFl_{G,\bN,\bV}$. The proofs are similar to the case of $\sA^\Gr_{G,\bN}$, which can be found in Section~1 of \cite{shiftoperatorsCoulombbranches}.

\begin{lem}\label{existresol}
There exists an $\BI \rtimes \Cxh$-equivariant resolution of singularities\footnote{That is, $\widetilde{C}_{\leq x}$ is nonsingular and $\rho_x$ is a proper birational morphism.}
\begin{equation*}
    \rho_x : \widetilde{C}_{\leq x} \longrightarrow C_{\leq x}
\end{equation*}
such that $\rho_x^{-1}(\BS|_{C_x})$ extends to a (necessarily unique) $\BI \rtimes \Cxh$-equivariant quotient vector bundle $\widetilde{\BS}_{\leq x}$ of $\rho_x^{-1}(\BT^d_{\leq x})$.
\end{lem}
We now fix the $\BI \rtimes \Cxh$-equivariant resolution $\rho_x : \widetilde{C}_{\leq x} \to C_{\leq x}$ and $\widetilde{\BS}_{\leq x}$ for each $x \in \Waffe$. It is easy to see that there is a proper surjective morphism
\begin{equation}\label{Stildesurjects}
\widetilde \BS_{\leq x}\twoheadrightarrow \rho_x^{-1}(\BS_{\leq x}).
\end{equation}

\begin{prop}\label{independentResol}
Let $x \in \Waffe$, and let $d$ be a sufficiently large positive integer such that $\BR^d_{\leq x}$ is well-defined. Then the following subspaces of $H^{\BI \rtimes \Cxh}_\bullet(\Fl_G)$ are equal:
\begin{enumerate}[label=\textup{(\arabic*)}]
    \item The sum of the images
    \begin{equation*}
        e(\widetilde{\BS}_{\leq y}) \cap H^{\BI \rtimes \Cxh}_\bullet(\widetilde{C}_{\leq y})
    \end{equation*}
    in $H^{\BI \rtimes \Cxh}_\bullet(\Fl_G)$ under the pushforward
    \begin{equation*}
        H^{\BI \rtimes \Cxh}_\bullet(\widetilde{C}_{\leq y})
        \xrightarrow{\rho_{y*}}
        H^{\BI \rtimes \Cxh}_\bullet(C_{\leq y})
        \subseteq
        H^{\BI \rtimes \Cxh}_\bullet(\Fl_G),
    \end{equation*}
    taken over all $y \leq x$;
    
    \item The image of $H^{\BI \rtimes \Cxh}_\bullet(\BR^d_{\leq x})$ under the composition
    \begin{equation*}
        H^{\BI \rtimes \Cxh}_\bullet(\BR^d_{\leq x})
        \longrightarrow
        H^{\BI \rtimes \Cxh}_\bullet(\BT^d_{\leq x})
        \stackrel{z_x^*}{\longrightarrow}
        H^{\BI \rtimes \Cxh}_\bullet(C_{\leq x})
        \subseteq
        H^{\BI \rtimes \Cxh}_\bullet(\Fl_G),
    \end{equation*}
    
    \item\label{item3} The direct sum
    \begin{equation*}
        \bigoplus_{y \leq x}
        \Q[\lt][\hbar] \cdot
        p_{y*}\!\left(
            e(\widetilde{\BS}_{\leq y}) \cap [\widetilde{C}_{\leq y}]
        \right),
    \end{equation*}
    i.e., the free $\Q[\lt][\hbar] \simeq H^\bullet_{T \times \Cxh}(\pt)$-submodule of $H^{\BI \rtimes \Cxh}_\bullet(\Fl_G)$ with basis
    \begin{equation*}
        \{\, p_{y*}( e(\widetilde{\BS}_{\leq y}) \cap [\widetilde{C}_{\leq y}] ) \,\}_{y \leq x}.
    \end{equation*}
\end{enumerate}

Moreover, each element $p_{y*}(e(\widetilde{\BS}_{\leq y}) \cap [\widetilde{C}_{\leq y}])$ is independent of the choice of resolution.
\end{prop}

We will simply denote $ p_{x*}( e(\widetilde{\BS}_{\leq x}) \cap [\widetilde{C}_{\leq x}] ) $ by $e(\BS) \cap [C_{\leq x}]$ from now on.

\subsubsection*{Flavor symmetry and deformations}
Suppose the \emph{flavor symmetry group} $F=(\Cx)^r$ acts linearly on $\bN$, preserving $\bV$, and commuting with the action of $G$.

We can define a deformed version of the Iwahori--Coulomb branches by replacing each $\BI$-equivariance by $\hBI\coloneq \BI \times F$-equivariance, and similarly replacing $G_\oh$-equivariance by $G_\oh \times F$-equivariance. All results discussed in this section generalize to this setting. For example, $H^{\BI \rtimes \Cxh \times F}_\bullet(\BR_{G,\bN,\bV})$ is a subalgebra of $H^{\BI \rtimes \Cxh \times F}_\bullet(\Fl_G)$ under the convolution product. 

For the rest of the paper, we mainly work with the $F$-deformed Coulomb branch, which we still denote by $\sAFl_{G,\bN,\bV}$ by abuse of notation.

\subsection{Preliminaries on quantum cohomology}\label{subsection_quantum_cohomology_and_quantum_connection}
Let $X$ be a smooth semiprojective variety equipped with an algebraic $\hG\coloneq G\times F$-action such that $X^{\hT}$ is compact. By semiprojective we mean that the affinization map
\[X\longrightarrow X^{\mathrm{aff}}\coloneq \spec H^0(X,\mathcal{O}_X)\]
is proper. 

In the following, we will define $\hT$-equivariant quantum cohomology of $X$ with $G$-equivariant Novikov variables $QH_{\hT}^\bullet(X)[[q_G]]$. We will also use the $\hG$-equivariant version, whose construction is similar. The variables $q_G$ are defined in order to record the curve class data in the definition of shift operators.

Define
\begin{equation*}
\Q[q_{G,X}] \coloneq \Q[q_G] \coloneq \Q[H_2^{G}(X; \Z)]
\end{equation*}
to be the group algebra of the abelian group $H_2^{G}(X; \Z)$.
We declare
\begin{equation*}
\deg q^\beta = 2\langle c_1^G(X), \beta \rangle.
\end{equation*}

Let $\Eff(X) \subseteq H_2(X; \Z)$ be the submonoid of effective curve classes.  
Denote by $\iota_*: H_2(X; \Z) \to H_2^G(X; \Z)$ the natural map.  
Define the completed tensor product $H_{\hT}^\bullet(X)[[q_G]]\coloneq H_{\TT}^\bullet(X)\widehat\otimes_{\Q} \Q[[q_G]]$ to be the graded completion of $H_{\hT}^\bullet(X) \otimes_\Q \Q[q_G]$ along the direction $\Eff(X)$.  
Concretely, an element of $H_{\hT}^\bullet(X)[[q_G]]$ is a formal sum $\sum_{\beta\in H_2^G(X;\Z)} q^\beta c_{\beta}$ satisfying
\begin{enumerate}
    \item each $c_{\beta}$ lies in $H_{\hT}^\bullet(X)$;
    \item there exists a finite subset $S \subseteq H_2^{G}(X; \Z)$ such that $c_{\beta} = 0$ unless $\beta \in S + \iota_*(\Eff(X))$; and
    \item there exists a finite subset $R \subseteq \Z$ such that $(c_{\beta})_d = 0$ for $d + \deg q^\beta \not\in R$, where $(c_{\beta})_d$ is the degree-$d$ component of $c_{\beta}$.
\end{enumerate}

\begin{df}\label{equivariantqh}
The \emph{$\hT$-equivariant quantum cohomology} of $X$ with $G$-equivariant Novikov variables is the ring with underlying vector space $QH^\bullet_{\hT}(X):=H^\bullet_{\hT}(X)[[q_G]]$, whose ring structure is defined by the \emph{quantum product} $\star$, where
\begin{equation*}
  \gamma \star \gamma'
  \coloneq \sum_{\beta \in \Eff(X)} 
    q^{\iota_* \beta}\,
    \mathrm{PD} \circ \ev_{3*}
    \Bigl(
        \ev_1^*(\gamma)\,
        \ev_2^*(\gamma')
        \cap
        [\overline{M}_{0,3}(X,\beta)]^\vir
    \Bigr),
\end{equation*}
for $\gamma, \gamma' \in H^\bullet_{\hT}(X)$.  
Here, $\overline{M}_{0,3}(X, \beta)$ is the moduli stack of genus-zero stable maps to $X$ with $3$ marked points and curve class $\beta$, and $[\overline{M}_{0,3}(X, \beta)]^\vir$ denotes its virtual fundamental class.  
Note that $\ev_3$ is proper, since $X$ is semiprojective.  
The product $\star$ on a general element of $H^\bullet_{\hT}(X)[[q_G]]$ is then defined termwise on its power series expansion.
\end{df}
\begin{rem}
The quantum product is usually defined using the non-equivariant Novikov parameters in $H_2(X;\Z)$. To see that the quantum product can also be defined using the equivariant Novikov parameters, we have to show that the kernel of the homomorphism $H_2(X; \Z)\to H_2^G(X; \Z)$ does not contain any effective curve class. In fact, let $L$ be a $G$-equivariant ample line bundle on $X$. Then $\langle c_1^G(L),[C]\rangle>0$ for any complete curve $C$ on $X$. This proves the claim.
\end{rem}
From now on, we write $q^{\iota_*\beta}$ simply as $q^\beta$ whenever there is no confusion.

We let the group of loop rotation $\Cxh$ act trivially on $X$. The following operator is the fundamental solution to the quantum $D$-module (\cite{IritaniFourier,Dubrovin,GiventalequivGW}). Since we will not need the notion, we omit the details about quantum connections.

\begin{df}[Fundamental solution]\label{def:fundamental_solution}
The \emph{fundamental solution} to the quantum differential equation is the operator
\begin{equation*}
\mathbb{M}_X :
H_{\hT\times\Cxh}^\bullet(X)[[q_G]][[\hbar^{-1}]]
\longrightarrow
H_{\hT\times\Cxh}^\bullet(X)[[q_G]][[\hbar^{-1}]]
\end{equation*}
defined by
\begin{equation*}
\mathbb{M}_X(\gamma) 
\coloneq \gamma
+ \sum_{0 \neq \beta \in \Eff(X)}
q^\beta \,
\mathrm{PD} \circ \ev_{2*}
\left(
    \frac{\ev_1^*(\gamma)}{\hbar - \psi_1} \,
    \cap 
    [\overline{M}_{0,2}(X,\beta)]^\vir
\right),
\end{equation*}
where $\psi_1$ is the equivariant first Chern class of the universal cotangent line bundle associated with the first marked point.
\end{df}

\subsection{Shift operators}\label{subsection_shift_operators}

Let $X$ be a smooth semiprojective variety satisfying the assumptions stated in \Cref{subsection_quantum_cohomology_and_quantum_connection}. For each $w\in W=N_G(T)/T$, we choose a representative $\dot w\in N_G(T)$. In this section, we regard a cocharacter $\lambda\in \Lambda$ as a group homomorphism $\lambda:\Cx\to T\hookrightarrow \TT$.

\begin{df}\label{seidelspace}
    Let $x=wt_\lambda\in \Waffe$. The \emph{Seidel space} of $X$ associated with $x\in \Waffe$ is the $X$-bundle over $\PP^1$
\begin{equation*}
    E_x\coloneq E_x(X)\coloneq \bigl(U_1 \bigsqcup U_2 \bigr) \big/ \sim,
\end{equation*}
where each of $U_1$ and $U_2$ is a copy of $X \times \C$, and a point $(p,z)\in U_1$ is identified with $(\dot w\lambda(z)\cdot p,\, z^{-1})\in U_2$ whenever $z\neq 0$.
The projection $E_x \to \PP^1$ is given by
\begin{equation*}
   (p_1,z_1)\in U_1 \longmapsto [z_1,1], 
   \qquad
   (p_2,z_2)\in U_2 \longmapsto [1,z_2].
\end{equation*}
\end{df}

There is a $\hT\times\Cxh$-action on $E_x$ given by
\begin{equation*}
\begin{aligned}
    (t,t_\hbar)\cdot(p_1,z_1) &\coloneq \bigl(\lambda(t_\hbar)\, w^{-1}(t)\cdot p_1,\; t_\hbar^{-1} z_1 \bigr),
    && (p_1,z_1)\in U_1, \\
    (t,t_\hbar)\cdot(p_2,z_2) &\coloneq \bigl(t\cdot p_2,\; t_\hbar z_2 \bigr),
    && (p_2,z_2)\in U_2.
\end{aligned}
\end{equation*}

Let $t\in T$. Suppose $E'_x$ is defined similarly to $E_x$ but using the representative $\dot w t$ of $w$.  
Then there is a $\hT\times \Cxh$-equivariant isomorphism $E'_x\xrightarrow{\sim} E_x$ given by  
$(p_1 ,z_1)\mapsto (t p_1,z_1)$ on $U_1$ and the identity on $U_2$.

We regard $X$ as a $\hT\times \Cxh$-space with trivial $\Cxh$-action. 

\begin{df}\label{df_twisting}
Let $X^x$ the variety $x$ equipped with the twisted $\hT\times \Cxh$-action:
\begin{equation*}
    (t,t_\hbar)\cdot p = \lambda(t_\hbar)\, w^{-1}(t)\cdot p.
\end{equation*}
The identity map $X\to X^x$ is equivariant with respect to the automorphism $\phi_x$ of $\hT\times\Cxh$ defined by
\begin{equation*}
    \phi_x:(t,t_\hbar) \longmapsto \bigl(w(\lambda(t_\hbar)^{-1} t),\, t_\hbar\bigr).
\end{equation*}
We denote by $\Phi_x$ the induced isomorphism on equivariant cohomology, i.e.,
\begin{equation}
    \Phi_x: H_{\hT\times\Cxh}^\bullet(X)\stackrel{\sim}{\to} H_{\hT\times\Cxh}^\bullet(X^x).
\end{equation}
\end{df}
By construction, $\Phi_x$ is not $\Q[\hat\lt][\hh]$-linear; instead, it satisfies the twisted linearity
\begin{equation}\label{twistedlinear}
    \Phi_x(f\gamma)=w(f^\lambda)\Phi_x(\gamma)
\end{equation}
for any $f\in \Q[\hat\lt][\hh]$ and $\gamma\in H_{\hT\times\Cxh}^\bullet(X)$, where $f^\lambda(\xi_,\hbar)=f(\xi+\lambda\hbar,\hbar)$. It extends uniquely to a twisted linear map
\[H_{\TT\times\Cxh}^\bullet(X)_\loc\to H_{\TT\times\Cxh}^\bullet(X^x)_\loc.\]
The inclusions $\iota_0 : X^x \to E_x$ and $\iota_\infty : X \to E_x$ defined by
\begin{equation*}
    \iota_0(p) = (p,0)\in U_1, 
    \qquad 
    \iota_\infty(p) = (p,0)\in U_2
\end{equation*}
are $\hT\times\Cxh$-equivariant.
Let $\beta\in H_2(E_x;\Z)$. Define $\cM_{x,\beta}(X)$ to be the stack
\[\cM_{x,\beta}(X)\coloneq \overline{M}_{0,2}(E_x,\beta)\times_{((\ev_1,\ev_2),(\iota_0,\iota_\infty))}(X^x\times X)\]
and denote the evaluation morphisms by 
\[\ev_0^\beta: \cM_{x,\beta}(X)\longrightarrow X^x,\quad \ev_\infty^\beta: \cM_{x,\beta}(X)\longrightarrow X.\]
The virtual fundamental class $[\cM_{x,\beta}(X,\beta)]^\vir$ is defined to be the refined Gysin pullback of $[\overline M_{0,2}(E_x,\beta)]^\vir$ along the inclusion $(\iota_0,\iota_\infty):X^x\times X\hookrightarrow E_x\times E_x$.

The projections from $U_i$ onto $X$ lie in the same $N_G(T)$-orbits and this defines a map $E_x\to [N_G(T)\backslash X]$. Denote by $\overline\beta$ the image of $\beta$ under the natural map
\[H_2(E_x;\Z)\longrightarrow H_2^{N_G(T)}(X;\Z)\longrightarrow H_2^G(X;\Z).\] 
An effective curve class $\beta\in \Eff(E_x)\subseteq H_2(E_x;\Z)$ is called a section class if its pushforward to $H_2(\PP^1;\Z)$ is $[\PP^1]$.

\begin{df}\label{df_ab_shift_op}
Let $x\in \Waffe$. The \emph{shift operators associated with $x$} is
\begin{equation*}
    \bbS_x: H_{\hT\times\Cxh}^\bullet(X)_\loc[[q_G]]\longrightarrow H_{\hT\times\Cxh}^\bullet(X)_\loc[[q_G]],
\end{equation*}
\begin{equation*}
    \bbS_x \coloneq S_x \circ \Phi_x,
\end{equation*}
where
\begin{equation*}
    S_x: H_{\hT\times\Cxh}^\bullet(X^x)_\loc[[q_G]]
        \longrightarrow
        H_{\hT\times\Cxh}^\bullet(X)_\loc[[q_G]]
\end{equation*}
is defined by
\begin{equation}\label{abshiftop}
    S_x(\gamma)
    \coloneq
    \sum_{\beta}
        q^{\overline\beta}\;
         \PD \Bigl(
        (\ev^\beta_{\infty})_*
        \Bigl(
            (\ev_{0}^{\beta})^*(\gamma)
            \cap 
            [\cM_{x,\beta}(X)]^\vir
        \Bigr) \Bigr),
\end{equation}
for $\gamma\in H_{\hT\times\Cxh}^\bullet(X)_\loc$. Here, the sum is taken over all section classes. In general, $S_x(\gamma)$ is defined termwise according to the expansion of $\gamma$ in $q_G$. Here, $\PD$ denotes the Poincar\'e duality isomorphism.
\end{df}

The induced $\hT\times\Cxh$-action on $\cM_{x,\beta}(X)$ has proper fixed locus and the pushforward $\ev^\beta_{\infty*}$ in \eqref{abshiftop} is defined using virtual localization formula. 

\begin{lem}\label{shiftoperatordegree}
    $\bbS_x$ preserves degrees.
\end{lem}
\begin{proof}
    It is clear that $\Phi_x$ preserves degrees. As for $S_x$, for any section class $\beta$, the virtual dimension of $\cM_{x,\beta}(X)$ is equal to $\dim X+\langle c_1(T^{\mathrm{vert}}E_x), \beta\rangle$, where $T^{\mathrm{vert}}E_x$ is the relative tangent bundle of $E_x\to \PP^1$. The claim then follows from the fact that $\deg(q^{\overline\beta})=2\langle c_1(T^{\mathrm{vert}}E_x),\beta\rangle$.
\end{proof}

Next, we discuss some generalities on twisted equivariant maps.
Let $K$ be a topological group and $\phi$ be an automorphism of $K$.
We say that a continuous map $f:Y\to Y'$ of $K$-spaces is
$\phi$-twisted equivariant if
\begin{equation*}
f(gy)=\phi(g)f(y)
\end{equation*}
for $g\in K$ and $y\in Y$.
Such a map induces a homomorphism
\begin{equation*}
f^*_\phi:H_K^\bullet(Y')\to H_K^\bullet(Y).
\end{equation*}
If $Y$ and $Y'$ are smooth semiprojective varieties, and
$f$ is an isomorphism, then $f^*_\phi$ intertwines the
Gromov--Witten invariants of $Y$ and $Y'$.

As an example, we have $\Phi_\xx=(\operatorname{id}_{\phi_\xx}^*)^{-1}$.
As another example, let $m_{\dot w}:X\to X$ be the multiplication morphism by $\dot w$.
Then the assignment
\begin{equation*}
w \longmapsto ((m_{\dot w})^*_w))^{-1} :
H_{\TT\times\Cxh}^\bullet(X)
\stackrel{\sim}{\longrightarrow}
H_{\TT\times\Cxh}^\bullet(X)
\end{equation*}
coincides with the (standard) $W$-action on
$H_{\TT\times\Cxh}^\bullet(X)$.
We denote $((m_{\dot w})^*_w)^{-1}$ simply by $w$.

\begin{lem}\label{lemma_S_x}
    Let $x=wt_\lambda\in W$, then $\bbS_x=w\circ \bbS_{t_\lambda}= \bbS_{t_{w(\lambda)}}\circ w$. In particular, $\bbS_w=w$ for any $w\in W$.
\end{lem}
\begin{proof}
Note that $\Phi_{t_\lambda}=\operatorname{id}_{w}^*\circ \Phi_\xx$. On the other hand, there is a $w$-twisted equivariant isomorphism $E_{t_{\lambda}} \stackrel{\sim}{\to} E_x$ defined by $(p_1,z_1)\mapsto (p_1,z_1)$ on $U_1$ and $(p_2,z_2)\mapsto (\dot w(p_2),z_2)$ on $U_2$, which implies $S_\xx=w\circ S_{t_\lambda}\circ\operatorname{id}_{w}^*$. Therefore, we have
\begin{equation*}
    \bbS_\xx=S_\xx\circ\Phi_\xx=w\circ S_\lambda\circ\operatorname{id}_{w}^*\circ \Phi_\xx=w\circ S_{t_\lambda}\circ\Phi_{t_\lambda}=w\circ\bbS_{t_\lambda}.
\end{equation*}
The equality $\bbS_x= \bbS_{t_{w(\lambda)}}\circ w$ can be proved similarly using the isomorphism $E_x \stackrel{\sim}{\to} E_{t_{w(\lambda)}}$ defined by $(p_1,z_1)\mapsto (\dot w(p_1),z_1)$ on $U_1$ and $(p_2,z_2)\mapsto (p_2,z_2)$ on $U_2$.
\end{proof}

\begin{prop}\label{localizedshiftop}
There is a graded algebra homomorphism
\begin{align*}
\bbS_{\loc}:\Q[\hat\lt][\hh]_\loc\rtimes\Waffe
\longrightarrow
\End\bigl(H_{\hT\times\Cxh}^\bullet(X)_{\loc}[[q_G]]\bigr),
\end{align*}
defined by
\begin{equation*}
x \longmapsto \bbS_x.
\end{equation*}
\end{prop}

\begin{proof}
The only thing that needs to be checked is
$\bbS_\xx\circ \bbS_\yy=\bbS_{\xx\yy}$.
It is proved in Theorem~5.12 of~\cite{shiftoperatorsCoulombbranches}
that $\bbS_{e}=\id$ and
$\bbS_{t_{\lambda+\mu}}=\bbS_{t_\lambda}\circ \bbS_{t_\mu}$.
Let $\xx=wt_\lambda$ and $\yy=ut_\mu$, then
$\xx\yy=wut_{u^{-1}(\lambda)+\mu}$, and hence
\begin{equation*}
\bbS_\xx\circ \bbS_\yy
= w\circ \bbS_{t_\lambda}\circ u\circ \bbS_{t_\mu}
= w\circ u\circ \bbS_{t_{u^{-1}(\lambda)}}\circ \bbS_{t_\mu}
= (wu)\circ \bbS_{t_{u^{-1}(\lambda)+\mu}}
= \bbS_{\xx\yy}. \qedhere
\end{equation*}
\end{proof}

\subsubsection*{Intertwining property}\label{localizationsection}
Let $\lambda\in \Lambda$.  
For simplicity, suppose $X^{\hT}=\{p_1,p_2,\dots,p_\ell\}$ is finite\footnote{This finiteness is not necessary; see Section~3.3 of \cite{Iritani} or Section~5.3 of \cite{shiftoperatorsCoulombbranches}.}.  
Define $\Delta_i(\lambda)\in H^\bullet_{\hT}(\pt)\otimes\Q[[q_G]]$ by
\begin{equation*}
    \Delta_i(\lambda)
    = q^{\beta_{p_i,\lambda}}
      \prod_{\alpha}
      \frac{\prod_{c=0}^{\infty}(\alpha+c\hbar)}
           {\prod_{c=\alpha(\lambda)}^{\infty}(\alpha+c\hbar)},
\end{equation*}
where the product runs over all $\TT$-weights $\alpha$ appearing in the weight decomposition of $T_{p_i}X$.  
Here, $\beta_{p_i,\lambda}\in H^G_2(X;\Z)$ denotes the image under $H_2(E_{t_\lambda};\Z)\to H_2^G(X;\Z)$ of the class of the constant section with value $p_i$ in both charts $U_1$ and $U_2$.

Alternatively, we may describe $\beta_{p_i,\lambda}$ as follows. 
We first identify $H^T_2(\pt;\Z)$ with the coweight lattice $\Lambda$ such that an element $\lambda\in \Lambda$ corresponds to a class in $H^T_2(\pt;\Z)$ whose pairing with $c_1^T(\C_\xi)$ is equal to $\langle\xi,\lambda\rangle$. Then $\beta_{p_i,\lambda}$ is the image of $\lambda$ under the homomorphism
\begin{equation*}
    \Lambda\simeq H_2^T(p_i;\Z)\longrightarrow H_2^T(X;\Z) \longrightarrow H^G_2(X;\Z),
\end{equation*}
where the middle map is the pushforward along the inclusion of fixed point $p_i\hookrightarrow X$.

\begin{thm}[\cite{BMO,Iritani}]\label{localizationthm}
There exists a commutative diagram
\begin{equation*}
\begin{tikzcd}[column sep=6 cm, row sep=large]
    H_{\hT\times\Cxh}^\bullet(X)_\loc
    \arrow[r, "\mathbb{M}^{-1} \circ \bbS_{t_\lambda} \circ \mathbb{M}"]
    \arrow[d]
  & H_{\hT\times\Cxh}^\bullet(X)_\loc[[q_G]]
    \arrow[d] \\
    H_{\hT\times\Cxh}^\bullet(X^T)_\loc
    \arrow[r, "\bigoplus_i \Delta_i(\lambda)\circ \Phi_{t_\lambda}"]
  & H_{\hT\times\Cxh}^\bullet(X^T)_\loc[[q_G]].
\end{tikzcd}
\end{equation*}
Here, the vertical arrows are restrictions to fixed points.
\end{thm}

\begin{exmp}\label{springercalc}
Suppose that $X$ is a smooth symplectic variety with a $G$-invariant symplectic form. Assume $X^T=\{p_1,\dots,p_\ell\}$.  
We only consider the $T\times\Cxh$-equivariance in this example.  
Then the virtual fundamental class $[\overline M_{0,2}(X,\beta)]^\vir$ vanishes unless $\beta=0$ (see \cite[Section 4.2]{BMO}), and hence the fundamental solution $\mathbb{M}$ is the identity map.  
For each fixed point $p_i$, the set of $T$-weights of $T_{p_i}X$ is given by 
\begin{equation*}
    \{\pm\alpha_{i,1}, \pm\alpha_{i,2}, \dots, \pm\alpha_{i,k}\}.
\end{equation*}
Let $\lambda\in \Lambda$ and $1\leq i\leq \ell$. Let $\gamma \in H_{T\times\Cxh}^\bullet(X)_\loc$ be the class whose restriction to $p_{i'}$ is equal to $\prod_{j=1}^k \alpha_{i,j}$ if $i'=i$ and zero otherwise.
Then, by~\Cref{localizationthm},
\begin{equation*}
    \bbS_{t_\lambda}(\gamma)
    =
    (-1)^{\sum_j \langle\alpha_{i,j},\lambda\rangle}
    q^{\beta_{p_i,\lambda}}\,
    \gamma.
\end{equation*}
\end{exmp}

\subsection{Moduli of parametrized sections}\label{subsection_moduli_of_parametrized_sections}
In this paper, we restrict our interests to the shift operators corresponding to nonlocalized homology classes of $\Fl_G$ and $\Gr_G$ and show that they can be realized via a parametrized version of Seidel spaces. This setup is required for the proof of \Cref{Theorem A}.

Recall the abelian shift operators $\bbS_\loc$ defined in \Cref{df_ab_shift_op} and \Cref{localizedshiftop}.
\begin{df}\label{dfCoulbranchaction}
    Define the $\Q[\f][\hbar]$-bilinear maps
    \begin{align*}
    \bbS^{\Fl}_X : &H_\bullet^{\hBI\rtimes\Cxh}(\Fl_G) \otimes_{\Q[\f][\hbar]} H^\bullet_{\hT\times\Cxh}(X)_\loc[[q_G]]\longrightarrow H^\bullet_{\hT\times\Cxh}(X)_\loc[[q_G]] \\
    \bbS^{\Gr}_X: &H_\bullet^{\hBI\rtimes\Cxh}(\Gr_G)\otimes_{\Q[\f][\hbar]} H_{\TT\times\Cxh}^\bullet(X)_\loc[[q_G]]\longrightarrow H_{\TT\times\Cxh}^\bullet(X)_\loc[[q_G]]
    \end{align*}
    by
    \begin{align*}
        \bbS^{\Fl}_X(\Gamma_1,\gamma) &\coloneq \bbS_\loc(\Gamma_1)(\gamma) \\
        \bbS^{\Gr}_X(\Gamma_2,\gamma) &\coloneq \bbS_\loc(\Gamma_2)(\gamma).
    \end{align*}
    where $\gamma\in H^\bullet_{\hT\times\Cxh}(X)_\loc[[q_G]]$, and
$\Gamma_1\in H_\bullet^{\hBI\rtimes\Cxh}(\Fl_G)$ and
$\Gamma_2\in H_\bullet^{\hBI\rtimes\Cxh}(\Gr_G)$
are viewed, via localization, as elements of $\Q[\hat\lt][\hh]_\loc\rtimes\Waffe$ and $\Q[\hat\lt,\hh]_\loc\rtimes\Lambda$, respectively.
\end{df}
Note that $\bbS^{\Fl}_X$ and $\bbS^{\Gr}_X$ are $\Q[\hat\lt][\hbar]$-linear with respect to the first argument. By \Cref{localizedshiftop}, $\bbS^{\Fl}_X$ defines a graded algebra action on $H_{\TT\times\Cxh}^\bullet(X)_\loc[[q_G]]$. We refer to $\bbS^{\Fl}_X$ the \emph{Iwahori--Coulomb branch action}. When there are no confusions, we simply denote the action of a class $\Gamma\in H_\bullet^{\hBI\rtimes\Cxh}(\Fl_G)$ by $\Gamma\cdot \gamma$. 

\begin{df}\label{dfseidelmap}
    Define the $\Q[\hat\lt][\hbar]$-linear map
    \[\Psi_X: H_\bullet^{\hBI\rtimes\Cxh}(\Gr_G) \longrightarrow H_{\TT\times\Cxh}^\bullet(X)_\loc[[q_G]]\]
    by
    \[\Psi_X(\Gamma)\coloneq \bbS^{\Gr}_X(\Gamma,1).\]
    We also set $\Psi_X^{\hbar=0}$ to be the specialization of $\Psi_X$ at $\hbar=0$.
\end{df}

Let $[g]\in \Fl_G$ be a point with representative $g\in G_\ok$.  
We can generalize the construction in the previous section to obtain a principal $G$-bundle $P_g$ over $\PP^1$.  
It is the unique principal $G$-bundle equipped with trivializations $\varphi_0$ and $\varphi_\infty$ over the charts $\spec \oh$ and $\spec \C[t^{-1}]$, respectively, such that the transition function $\varphi_\infty\circ \varphi_0^{-1}$ is given by multiplication by $g\in G_\ok$.
The existence and uniqueness follow from a result of \citeauthor{BL} \cite{BL}.  
Note that if $[g']=[g]$, and $P_{g'}$ is the corresponding principal $G$-bundle on $\PP^1$ with trivializations $\varphi'_0$ and $\varphi'_\infty$, then there is an isomorphism
\[\psi : E_{g'} \simeq E_g\]
with $\varphi'_\infty = \varphi_\infty \circ \psi$, and $\varphi'_0|_{t=0}\, b = \varphi_0\circ \psi|_{t=0}$ for some $b\in B^-$.

As a result, the affine flag variety $\Fl_G$ represents the functor that sends a scheme $S$ to the isomorphism classes of tuples $(P,\varphi,r)$ where $P$ is a principal $G$-bundle on $S\times \PP^1$, $\varphi : P|_{S\times (\PP^1\setminus 0)} \to S\times G\times (\PP^1\setminus 0)$ is a trivialization, and $r$ is a section of $P/B^-$ over $S\times\{0\}$.  
We refer the reader to \cite{Zhu} for more details.

Let $x=wt_\lambda\in \Waffe$. Take a resolution $\rho_x:\widetilde C_{\leq x}\to C_{\leq x}$.  
The composition $\widetilde C_{\leq x}\stackrel{\rho_x}{\to}C_{\leq x}\hookrightarrow \Fl_G$ defines a principal $G$-bundle $P_{\leq x}$ on $\widetilde C_{\leq x}\times \PP^1$, and we write
\begin{equation*}
    E_{\leq x}\coloneq E_{\leq x}(X) \coloneq P_{\leq x}\times_G X.
\end{equation*}

Let $\widetilde G_\ok^{\leq x}$ denote the fiber product $G_\ok\times_{\Fl_G}\widetilde C_{\leq x}$. Define
\begin{equation*}
    \iota_0:\mathcal X_{\leq x,0}\coloneq \widetilde G_\ok^{\leq x}\times_\BI X \to E_{\leq x},
    \qquad
    \iota_\infty:\mathcal X_{\leq x,\infty}\coloneq \widetilde C_{\leq x}\times X \to E_{\leq x}
\end{equation*}
to be the inclusions of the fibers over $0$ and $\infty$.

Set $\hBI\coloneq\BI\times F$. For $\Gamma\in H^{\hBI\rtimes\Cxh}_\bullet(\widetilde C_{\leq x})$, we define a twisting map
\begin{equation*}
    \tw_\Gamma: H^\bullet_{\hT\times\Cxh}(X)\to H^\bullet_{\hBI\rtimes\Cxh}(\widetilde G^{\leq x}_\ok\times_\BI X)
\end{equation*}
so that $\tw_\Gamma(\gamma)$ is the image of $\PD(\Gamma)\otimes \gamma$ under the composition
\begin{equation*}
\begin{aligned}
H^\bullet_{\hBI\rtimes\Cxh}(\widetilde C_{\le x})
\otimes_{\Q[\f][\hbar]}
H^\bullet_{\hT\times\Cxh}(X)
&\;\simeq\;
H^\bullet_{(\hBI\times\BI\times T)\rtimes\Cxh}
\bigl(\widetilde G_\ok^{\le x}\times X\bigr) \\
&\rightarrow
H^\bullet_{(\hBI\times\Delta_{\BI})\rtimes\Cxh}
\bigl(\widetilde G_\ok^{\le x}\times X\bigr) \\
&\;\simeq\;
H^\bullet_{\hBI\rtimes\Cxh}
\bigl(\widetilde G_\ok^{\le x}\times_{\BI} X\bigr).
\end{aligned}
\end{equation*}

Let $\beta\in H_2(E_{\leq x};\Z)$. Define $\cM_{\leq x,\beta}(X)$ to be the stack 
\[\cM_{\leq x,\beta}(X)\coloneq \overline M_{0,2} (E_{\leq x},\beta) \times_{((\ev_1,\ev_2),(\iota_0,\iota_\infty))} (\mathcal X_{\leq x,0}\times \mathcal X_{\leq x,\infty}).\]
and denote the evaluation morphisms by
\begin{equation*} 
\ev^\beta_0 : \cM_{\leq x,\beta}(X)\longrightarrow \mathcal X_{\leq x,0}, \qquad \ev^\beta_\infty : \cM_{\leq x,\beta}(X)\longrightarrow \mathcal X_{\leq x,\infty}. 
\end{equation*}
The virtual fundamental class $[\cM_{\leq x,\beta}(X)]^\vir$ is defined to be the refined Gysin pullback of $[\overline M_{0,2}(E_{\leq x},\beta)]^\vir$  along 
\[(\iota_0,\iota_\infty): \mathcal X_{\leq x,0}\times\mathcal X_{\leq x,\infty} \longhookrightarrow E_{\leq x}\times E_{\leq x}.\]
See ~\cite[p. 23]{shiftoperatorsCoulombbranches} for more details. The projection onto the second factor defines a map $E_{\leq x}= P_{\leq x}\times_G X\to [G\backslash X]$. We denote by $\overline\beta$ the image of $\beta$ under the natural map
\[H_2(E_{\leq x};\Z)\longrightarrow H_2^G(X;\Z).\]
An effective curve class $\beta\in \Eff(E_{\leq x})\subseteq H_2(E_{\leq x};\Z)$ is called a section class if its pushforward to $H_2(\widetilde C_{\leq x}\times \PP^1;\Z)$ is $[\pt\times\PP^1]$.

Let $\pr: \widetilde C_{\leq x}\times X\to X$ be the projection onto $X$. We define a $\Q[\f][\hbar]$-bilinear map
\begin{equation*}
    \bbS_{\leq x} : H^{\hBI\rtimes\Cxh}_\bullet(\widetilde C_{\leq x}) \otimes_{\Q[\f][\hbar]} H^\bullet_{\hT\times\Cxh}(X) [[q_G]]
    \longrightarrow H^\bullet_{\hT\times\Cxh}(X))_\loc[[q_G]]
\end{equation*}
by
\[
\bbS_{\leq x}(\Gamma, \gamma) \coloneq S_{\leq x}(\tw_\Gamma(\gamma)),
\]
where
\begin{equation}\label{nonabshiftop}
    S_{\leq x}(\gamma') 
    \coloneq \sum_{\beta}
      q^{\overline\beta}\;
      \PD\Big( \pr_* 
      (\ev^\beta_{\infty})_*
      \Bigl(
        (\ev_{0}^{\beta})^*(\gamma')
        \cap [\cM_{\leq x,\beta}(X)]^\vir
      \Bigr)\Bigr).
\end{equation}
Here, the sum is taken over all section classes. In general, $S_{\leq x}(\gamma')$ is defined termwise according to the expansion of $\gamma'$ in $q_G$. The induced $\hT\times\Cxh$-action on $\cM_{\leq x,\beta}(X)$ has proper fixed locus and the pushforward $(\ev^\beta_{\infty})_*$ in \eqref{nonabshiftop} is defined using virtual localization formula. By~\cite[Lemma 4.6]{shiftoperatorsCoulombbranches}, the expression in \eqref{nonabshiftop} lies in $H_{\hT\times\Cxh}^\bullet(X)_\loc[[q_G]]$.

\begin{prop}\label{abelianizationprop}
    Let $\Gamma\in H_\bullet^{\hBI\rtimes\Cxh}(\widetilde C_{\leq x})$ and $\gamma\in H_{\TT\times\Cxh}^\bullet(X)[[q_G]]$, then we have
    \begin{equation}\label{localizeS^Fl}
        \bbS^{\Fl}_X((\rho_x)_*(\Gamma),\gamma)=\bbS_{\leq x}(\Gamma,\gamma).
    \end{equation}
\end{prop}
\begin{proof}
This is similar to Theorem~5.6 in \cite{shiftoperatorsCoulombbranches}, and we only give a sketch here.
Since both sides of \eqref{localizeS^Fl} are $\Q[\hat\lt][\hbar]_\loc$-linear in $\Gamma$, we may assume $\Gamma \in H_\bullet^{\hBI\rtimes\Cxh}(Z)$ for some fixed component $Z\subseteq \widetilde C_{\leq x}$. So $(\rho_x)_*(\Gamma)=c_y[y]$ for $\rho_x(Z)=\{y\}\subseteq (\Fl_G)^T$. We will use the symbol $y$ to denote the corresponding element in $\Waffe$.

It suffices to prove that for each section class $\beta\in H_2(E_{y};\Z)$ and $\gamma\in H^\bullet_{\hT\times\Cxh}(X)$,
\begin{equation*}
    \pr_*
    (\ev^\beta_{\infty})_*
      \Bigl(
        (\ev_{0}^{\beta})^*\bigl(\tw_\Gamma(\gamma)\bigr)
        \cap [\cM_{\leq x,\beta}(X)]^\vir
      \Bigr)
    =
    c_{y}
    (\ev^\beta_{\infty})_*
      \Bigl(
        (\ev_{0}^{\beta})^*\bigl(\Phi_{y}(\gamma)\bigr)
        \cap [\cM_{y,\beta}(X)]^\vir
      \Bigr).
\end{equation*}
If $\beta$ is not in the image of $H_2(E_{\leq x}|_F;\Z)$, then both sides vanish.  
If $\beta\in H_2(E_{\leq x}|_F;\Z)$, consider the fiber diagram
\begin{equation*}
\begin{tikzcd}
    Z\times \cM_{y,\beta}(X) \ar[r]\ar[d] &
    \cM_{\leq x,\beta}(X) \ar[d] \\
    Z \ar[r,"j"] &
    \widetilde C_{\leq x}.
\end{tikzcd}
\end{equation*}
Virtual localization gives
\begin{equation*}
    j^! \bigl[\cM_{\leq x,\beta}(X)\bigr]^\vir
    = [Z]\otimes [\cM_{y,\beta}(X)]^\vir.
\end{equation*}
The desired result then follows from this together with the fact that
\begin{equation*}
    \tw_{\Gamma}(\gamma)
    = i_*(\PD(\Gamma)\otimes\Phi_{y}(\gamma))
\end{equation*}
for $\gamma\in H^\bullet_{\hT\times\Cxh}(X)$, where $i: Z\times X^y\hookrightarrow\mathcal X_{\leq x,0}$ (see~\cite[Proposition 2.8]{shiftoperatorsCoulombbranches}).
\end{proof}

We collect some useful results about $\bbS_X^{\Fl}$ and $\bbS_X^{\Gr}$. Recall that there is an action $\cdot_{\Gr}$ of $H_\bullet^{\hBI\rtimes\Cxh}(\Fl_G)$ on $H^{\hBI_\oh\rtimes\Cxh}_\bullet(\Gr_G)$ (see \Cref{FlactsonGr}).
\begin{prop}\label{Results_of_CCL}
    The maps $\bbS_X^{\Fl}$ and $\bbS_X^{\Gr}$ satisfy the following properties.
    \begin{enumerate}
        \item $\bbS^{\Fl}_X$ defines a graded algebra action on $H_{\hT\times\Cxh}^\bullet(X)_\loc[[q_G]]$.
        \item For any $\Gamma\in H_\bullet^{\hBI\rtimes\Cxh}(\Fl_G)$, $\Gamma'\in H^{\hBI\rtimes\Cxh}_\bullet(\Gr_G)$, and $\gamma\in H_{\hG\times\Cxh}^\bullet(X)$, we have 
        \[\bbS^\Fl_X(\Gamma, \bbS^\Fl_X(\Gamma', \gamma))=\bbS^\Gr_X((\Gamma\cdot_{\Gr} \Gamma'), \gamma).\]
        \item $\Psi_X$ is $W$-equivariant. $\Psi_X^{\hbar=0}$ is a graded algebra homomorphism.
    \end{enumerate}
\end{prop}
\begin{proof}
    (1) follows from \Cref{shiftoperatordegree} and \Cref{localizedshiftop}. The proof of (2) follows from a similar argument to that of~\cite[Theorem 5.13]{shiftoperatorsCoulombbranches}. (3) is Corollary 5.10 and Corollary 5.24 in \textit{op. cit.}
\end{proof}

\subsection{\texorpdfstring{Proof of \Cref{Theorem A}}{Proof of Theorem A}}\label{subsection_proof_of_BFN_action}
Let $\bN$ be a finite-dimensional $G$-representation, and $\bV\subseteq \bN$ a $B^-$-invariant subspace. In what follows, $\sAFl_{G,\bN,\bV}$ and $\sAGr_{G,\bN}$ denote the deformed Iwahori--Coulomb branch and Coulomb branch algebras, respectively. Here, the flavor symmetry group $\Cxd$ acts on $\bN$ with positive weights and commutes with the $G$-action.

Suppose there exists a $\hG$-equivariant proper morphism $f:X\to\bN$. 

\begin{thm}[=\Cref{Theorem A}]\label{maintext_thm_BFN_action}
The shift operators induce a graded algebra action
\begin{equation*}
    \bbS^{\Fl}_X: H^{\hBI\rtimes\Cxh}_\bullet(\Fl_G) \otimes_{\Q[\f][\hbar]} QH^\bullet_{\TT\times\Cxh}(X)_{\loc} \to QH^\bullet_{\TT\times\Cxh}(X)_{\loc}
\end{equation*}
such that
\[(\Gamma \cdot \gamma, \gamma')^{\TT}_{X} \in \Q[\hat\lt][\hbar][[q_G]]\]
for any $\Gamma \in \sAFl_{G,\bN,\bV}$ and any $\gamma, \gamma'\in H^\bullet_{\TT\times\Cxh}(X)$ satisfying $\supp(\gamma) \subseteq f^{-1}(\bV)$ and $\supp(\gamma') \cap f^{-1}(\bV) \subseteq f^{-1}(0)$.
\end{thm}

Let $x\in \Waffe$. We fix an $\hBI\rtimes\Cxh$-equivariant resolution $\rho_x:\widetilde C_{\leq x}\to C_{\leq x}$ such that $\BS_{G,\bN,\bV}|_{C_{x}}$ extends to a vector bundle on $\widetilde C_{\leq x}$ as in~\Cref{existresol}.

\begin{lem}\label{Lemcansec}
There exists a canonical section $\can$ of $\widetilde\BS_{\leq x}$ over 
$(\ev^\beta_0)^{-1}(\widetilde G_\ok^{\leq x}\times_\BI f^{-1}(\bV))\subseteq \cM_{\leq x,\beta}(X)$.  
Moreover, the restriction of $\ev^\beta_\infty$ to the zero locus of $\can$ is proper.
\end{lem}

\begin{proof}
Let $f':E_{\leq x}(X)\to E_{\leq x}(\bN)$ be the morphism defined by $[p,x]\mapsto [p,f(x)]$. It induces a proper morphism $\cM_{\leq x,\beta}(X)\to\cM_{\leq x,f'_*\beta}(\bN)$, so we may assume $X=\bN$ (c.f. ~\cite[Lemma 4.4]{shiftoperatorsCoulombbranches}). 
 
A stable map $\sigma$ representing a point in $(\ev^\beta_0)^{-1}(\widetilde G^{\leq x}_\ok\times_\BI \bV)$ has domain given by $(\PP^1,0,\infty)$. Since the curve class maps to $[\pt]\times [\PP^1]$, there is a point $p\in \widetilde C_{\leq x}$ such that $\sigma$ is a section of the vector bundle $E_p(\bN)=E_{\leq x}(\bN)\times_{\widetilde C_{\leq x}}\{p\}$, satisfying $\sigma(0)\in \widetilde G^{\leq x}_\ok\times_{\BI}\bV$. Let $\widetilde g\in \widetilde G^{\leq x}_\ok$ be a lift of $p\in \widetilde C_{\leq x}$ and denote by $g$ the image of $\widetilde g$ under $\widetilde G^{\leq x}_\ok\to G_\ok$. The data of $\sigma$ is then represented by
\begin{align*}
    \sigma_0 &\in \bV\oplus t\bN_{\C[[t]]},\\
    \sigma_\infty &\in \bN_{\C[t^{-1}]},
\end{align*}
with $\sigma_\infty = g \sigma_0$. Thus, $\sigma_\infty$ lies in $\BT_p$, descending to a point $\can(\sigma)\in\widetilde\BS_p$. Here, $\BT_p$ and $\widetilde\BS_p$ denote the fibers of $\rho_x^{-1}\BT$ and $\widetilde \BS_{\leq x}$ over $p$, respectively. This defines a morphism $\can : (\ev^\beta_0)^{-1}(\widetilde G^{\leq x}_\ok\times_\BI \bV)\to \widetilde\BS_{\leq x}$.

Denote $\cZ_\beta\subseteq (\ev_0^\beta)^{-1} (\widetilde G^{\leq x}_\ok\times_\BI\bV)$ the zero locus of $\can$.
By \eqref{Stildesurjects}, if $\sigma\in \cZ_\beta$, then $\sigma_\infty\in \bN_{\C[t^{-1}]}\cap (\bV\oplus t\bN_{\C[[t]]})= \bV$. Thus, the section $\sigma$ is constant with value in $\bV$ with respect to the trivialization $P_{\leq x}|_{\widetilde C_{\leq x}\times (\PP^1\setminus 0)}\times_G \bN \to \widetilde C_{\leq x}\times \bN\times (\PP^1\setminus 0)$ induced by $\varphi$ (see the second paragraph of \Cref{subsection_moduli_of_parametrized_sections}). 

We consider the diagram
\begin{equation*}
\begin{tikzcd}
    \cZ_\beta\ar[r]\ar[d,"\ev_\infty"] &
    \cM_{\leq x, \beta}(\bN)\ar[d] \\
    \widetilde C_{\leq x}\times \bV\ar[rd] &
    \widetilde G^{\leq x}_\ok\times_\BI \bN_\oh\ar[d] \\
    & \widetilde C_{\leq x}\times \bN_\ok.
\end{tikzcd}
\end{equation*}
All arrows except the left vertical one are closed immersions; hence the evaluation morphism $\ev_\infty$ restricted to $\cZ$ is also a closed immersion.
\end{proof}

\begin{proof}[Proof of \Cref{Theorem A}]
The claim that $\bbS^{\Fl}_X$ defines a graded algebra action is \Cref{Results_of_CCL}(1).

For the polynomiality property, by \Cref{independentResol}, for any $\Gamma\in \sAFl_{G,\bN,\bV}$, there exist $x\in \Waffe$ and $\Gamma'\in H^{\hBI\rtimes\Cxh}_\bullet(\widetilde C_{\leq x})$ such that $\Gamma=e(\widetilde\BS_{\leq x})\cap \Gamma'$. Then we have $\tw_\Gamma(-)=e(\widetilde\BS_{\leq x})\cup \tw_{\Gamma'}(-)$. Fix a section class $\beta\in H_2(E_{\leq x};\Z)$. For $\gamma\in H_{\TT\times\Cxh}^\bullet(X)$ satisfying $\supp(\gamma)\subseteq f^{-1}(\bV)$, we have $\tw_{\Gamma'}(\gamma)=j_*(\eta)$ for some $\eta\in H_{\TT\times\Cxh}^\bullet\bigl(\widetilde G_{\ok}^{\leq x}\times_{\BI} f^{-1}(\bV)\bigr)$, where $j:\widetilde G_{\ok}^{\leq x}\times_{\BI} f^{-1}(\bV)\hookrightarrow \mathcal X_{\leq x,0}$.

Let $\cZ_\beta$ be the zero locus of $\can$ as in \Cref{Lemcansec}. We define $[(\ev^\beta_0)^{-1}(\widetilde G^{\leq x}_\ok\times_\BI f^{-1}(\bV))]^\vir$ to be the refined Gysin pullback of $[\cM_{\leq x,\beta}(X)]^\vir$ along $\widetilde G^{\leq x}_\ok\times_\BI \bV\hookrightarrow \widetilde G^{\leq x}_\ok\times_\BI \bN$. And define $[\cZ_\beta]^\vir$ as the refined Gysin pullback of $[(\ev^\beta_0)^{-1}(\widetilde G^{\leq x}_\ok\times_\BI f^{-1}(\bV))]^\vir$ along the zero section $\widetilde C_{\leq x}\to \widetilde\BS_{\leq x}$. Then, by excess intersection formula and projection formula, the contribution from $\beta\in H_2(E_{\leq x};\Z)$ in
\[\bigl(\Gamma\cdot \gamma, \gamma'\bigr)^{\hT}_X\]
is  equal to
\[q^{\overline\beta}\int_{[\cZ_\beta]^\vir}(\ev^\beta_0)^*(\eta)\cup (\pr\circ \ev^\beta_\infty)^*(\gamma').\]
It suffices to show that 
\begin{equation*}
    (\ev_0^\beta)^{-1}(\widetilde G^{\leq x}_\ok\times_\BI \supp(\gamma))
    \cap
    (\pr\circ \ev^\beta_\infty)^{-1}(\supp(\gamma'))
\end{equation*}
is complete, which follows from \Cref{Lemcansec} and the assumptions that $\supp(\gamma)\subseteq f^{-1}(\bV)$ and $\supp(\gamma')\cap f^{-1}(\bV)\subseteq f^{-1}(0)$. This completes the proof.
\end{proof}

\subsection{The case of symplectic resolutions}\label{subsection_symplectic_resolutions}
\begin{df}
A smooth symplectic variety $X$ is called a \emph{conical symplectic resolution} if the affinization map
\begin{equation*}
    X \longrightarrow \spec H^0(X,\oh_X)
\end{equation*}
is projective and birational, and there exists a conical $\Cxd$-action on $X$ such that the symplectic form has negative weight. Here, conical means that $H^0(X,\mathcal O_X)$ has nonpositive $\Cxd$-weights and $H^0(X,\mathcal O_X)^{\Cxd}=\C$. We will take $F=\Cxd$ and write $\Q[\hat\lt]=\Q[\lt][k]$ from now on.
\end{df}

Let $\operatorname{Hamil}_{\Cxd}(X)$ be the group of symplectomorphisms commuting with the $\Cxd$-action; it is always a linear algebraic group\footnote{The action is always Hamiltonian.}.
Let $G\subseteq \operatorname{Hamil}_{\Cxd}(X)$ be a (maximal) reductive subgroup, and let $T\subseteq G$ be a maximal torus.

We assume for simplicity that $X^T$ is finite; but results below hold true if we only assume $X^T$ is complete. We now recall the stable basis for equivariant quantum cohomology introduced in \cite{MO}.
Let $\tau\in \Lambda$ be a generic cocharacter in the sense that $X^\tau = X^T$.

\begin{df}
For each $p\in X^T$, let
\begin{equation*}
    \operatorname{Attr}(p) := \{x\in X \mid \lim_{z\to 0}\tau(z)x = p\}.
\end{equation*}
There is a partial order on $X^T$ given by $p\geq q$ if and only if $q\in \overline{\operatorname{Attr}(p)}$.
We write $\operatorname{Attr}^f(p) = \bigcup_{q\leq p} \operatorname{Attr}(q)$.
\end{df}

\begin{df}
A \emph{polarization} $\epsilon$ of $X$ is a choice of $\epsilon_p\in H_T^\bullet(\pt)$ for each $p\in X^T$ such that
\begin{equation*}
    \epsilon_p^2 = (-1)^{\dim X/2} e^T(T_pX).
\end{equation*}
\end{df}

We now fix a polarization $\epsilon$. We denote by $N^-_p$ the negative $\tau$-weight subspace of $T_pX$.

\begin{thm}[Section~3.7 of \cite{MO}]\label{MO_stable_envelope}
There exists a unique $H_{\hT}^\bullet(\pt)$-linear map
\begin{equation*}
    \stab_\tau : H_{\hT}^\bullet(X^T) \longrightarrow H_{\hT}^\bullet(X)
\end{equation*}
such that for each fixed point $p\in X^T$ and each $\gamma\in H^\bullet_{\hT}(p)$, the following conditions hold:
\begin{enumerate}
    \item $\operatorname{Supp}(\stab_\tau(\gamma)) \subseteq \operatorname{Attr}^f(p)$;
    \item $\stab_\tau(\gamma)|_p = \pm e^{\hT}(N_p^-)\cup\gamma$, where the sign is determined by the condition
    $\stab_\tau(\gamma)|_p \equiv \epsilon_p\cup \gamma \pmod k$;
    \item $\stab_\tau(\gamma)|_q \equiv 0 \pmod k$ for all $q\neq p$.
\end{enumerate}
\end{thm}

The map $\stab_\tau$ is called the \emph{stable envelope}. For each $p\in X^T$, we identify $p$ with the generator of $H^\bullet_{\hT}(p)$. The collection $\{\stab_\tau(p)\}_{p\in X^T}$ forms a $\Q[\lt][k]_\loc$-basis of $H_{\hT}^\bullet(X)_\loc$, called the \emph{stable basis}.

\begin{rem}
    In \cite{MO}, stable envelopes are defined using chambers $\mathfrak C$ in $\lt_\R$, which are cut out by hyperplanes dual to the $T$-weights of normal bundles of $X^T$. It is easy to show that stable envelopes do not depend on the particular choice of $\tau\in \mathfrak{C}$.
\end{rem}

It is proved in Theorem 4.4.1 of \cite{MO} that
\begin{equation}\label{orthostabbasis}
    ( \stab_{-\tau}(p),\, (-1)^{\dim X/2}\stab_{\tau}(q))^{\hT}_X
    = \delta_{p,q},
\end{equation}
where $\stab_{-\tau}$ is the stable envelope corresponding to the cocharacter $-\tau$.

In the following, we fix a cocharacter $\tau$ such that $\g^\tau = \lt$ and the Lie algebra $\mathfrak{n}$ of $N\subseteq B\subseteq G$ is equal to the $\tau$-positive subspace of $\g$. Equivalently, $\tau$ is a generic element of the dominant Weyl chamber. We denote by $\stab^+=\stab_\tau$ and $\stabm=\stab_{-\tau}$. 

We choose a $\hG$-equivariant finite morphism (e.g., a closed embedding)
\begin{equation*}
    \spec H^0(X,\oh_X)\to \bN
\end{equation*}
to a $\hG$-representation $\bN$ with $\bN^{\Cxd}=0$. By the conical assumption on $X$, such $\bN$ can always be taken to have positive $\Cxd$-weights. Write $f:X\to \bN$ for the composition of the affinization map of $X$ with this morphism.
We fix the decomposition
\begin{equation*}
    \bN = \bV^+ \oplus \bV,
\end{equation*}
where $\bV^+$ and $\bV$ are the $\tau$-nonnegative and $\tau$-negative subspaces, respectively. In particular, $\bV$ is $B^-$-invariant. 

Consider the Iwahori--Coulomb branch $\sAFl_{G,\bN,\bV}$.
\begin{cor}[=\Cref{thm_BFN_action_symplectic_resolution}]\label{maintext_thm_BFN_action_symplectic_resolution}
We have
\begin{equation}\label{shiftstablebasis}
    \bigl(\Gamma\cdot\stabm(\gamma),\,\stabp(\gamma')\bigr)^{\hT}_X
    \in \RR[[q_G]]
\end{equation}
for any $\gamma, \gamma' \in H^\bullet_{\TT}(X^{T})$, and any $\Gamma$ in the deformed Iwahori--Coulomb branch $\sAFl_{G,\bN,\bV}$ with respect to the flavor symmetry $\Cxd$.
In particular, if $\deg \Gamma=0$, then
\begin{equation*}
    \bigl(\Gamma\cdot\stabm(p),\,\stabp(q)\bigr)^{\hT}_X
    \in \Q[[q_G]]
\end{equation*}
for any $p,q\in X^T$.
\end{cor}

\begin{proof}
By the linearity of the stable envelopes, we may assume $\gamma=p$, $\gamma'=q$ for $p,q\in X^T$. By the support condition of the stable envelopes, we have
\begin{equation*}
    \operatorname{Supp}(\stabm(p)) \subseteq f^{-1}(\bV),
    \qquad
    \operatorname{Supp}(\stabp(q)) \subseteq f^{-1}(\bV^+).
\end{equation*}
The first assertion now follows from \Cref{maintext_thm_BFN_action}. The second assertion follows from the fact that the $\deg\bigl(\Gamma\cdot\stabm(p),\,\stabp(q)\bigr)^{\hT}_X=\deg\Gamma$.
\end{proof}

\begin{cor}
The algebra $\sAFl_{G,\bN,\bV}$ preserves the image of $\stabm$.
\end{cor}
\begin{proof}
This follows from \Cref{maintext_thm_BFN_action_symplectic_resolution} and \Cref{orthostabbasis}.
\end{proof}

\section{Part II: Case of cotangent bundles of flag varieties and applications}\label{section_part2}
In this section, we apply the results of \Cref{section_part1} to the case where $\bN=\g^*,\bV=(\bb^-)^\perp$, and $X=T^*\PPP$ is the cotangent bundle of a flag variety $\PPP=G/P$. When $P=B$, we denote the full flag variety by $\BBB=G/B$. We let the flavor symmetry group $F=\Cxd$ act by dilation (of weight $1$) on the cotangent fibers and $\g^*$. To simplify the notation, we write
\begin{equation*}
    \sA := \sA^{\Fl}_{G,\g^*,(\bb^-)^{\perp}},
    \qquad
    \sAs := \sA^{\Gr}_{G,\g^*}.
\end{equation*}
Recall that the (Iwahori--)Coulomb branch algebras above are defined with respect to the \emph{flavor symmetry group $\Cxd$} (see the paragraph at the end of \Cref{subsection_affine_flag_coulomb_branch}).

\subsection{Iwahori--Coulomb branch for the coadjoint representation}\label{subsection_construction_of_a_basis}
Recall from \Cref{independentResol} that $\sA$ has a $\Q[\lt][\hbar,\dd]$-basis $\left\{b_\xx=e(\BS_{G,\g^*,(\bb^-)^{\perp}})\cap[C_{\leq \xx}]\right\}_{\xx\in\Waffe}$. For simplicity, put $\BS:=\BS_{G,\g^*,(\bb^-)^{\perp}}$.

\begin{lem}\label{lemma_fiber_of_BFN_bundle}
Let $\xx\in\Waffe$.
We have
\begin{equation*}
    \BS|_{\xx}
    \simeq T^*_{\xx}C_{\xx}
\end{equation*}
as $\Cxd\times\bigl((\xx\BI\xx^{-1}\cap\BI)\rtimes\Cxh\bigr)$-modules,
where $\Cxd$ acts on the second module by the weight $1$ action.
In particular, $b_{\xx}$ has degree $0$, and
\begin{equation*}
    b_{\xx}
    \equiv (-1)^{\ell(\xx)}[\xx] + \sum_{\yy<\xx}c_{\xx,\yy}[y] \pmod{\dd}
\end{equation*}
for some $c_{\xx,\yy}\in\Q[\lt][\hbar]_{\loc}$.
\end{lem}
\begin{proof}
By definition,
\begin{equation*}
    \BS|_{\xx} \simeq
    \xx((\bb^-)^{\perp}\oplus t\g^*[[t]])\big/
    \bigl(\xx((\bb^-)^{\perp}\oplus t\g^*[[t]])
    \cap ((\bb^-)^{\perp}\oplus t\g^*[[t]])\bigr),
\end{equation*}
and
\begin{equation*}
    T_{\xx}^*C_{\xx} \simeq
    \bigl(\Lie(\BI)\big/\bigl(\Lie(\BI)\cap \xx\cdot\Lie(\BI)\bigr)\bigr)^*,
\end{equation*}
where $\Lie(\BI) := \bb^-\oplus t\g[[t]]$.

To show that these
$\Cxd\times\bigl((\xx\BI\xx^{-1}\cap\BI)\rtimes\Cxh\bigr)$-modules are isomorphic,
consider the linear map
\begin{equation*}
    \Lie(\BI)\otimes \xx((\bb^-)^{\perp}\oplus t\g^*[[t]]) \longrightarrow \C
\end{equation*}
defined by sending $s_1\otimes s_2$ to the constant term of $\langle s_1,s_2\rangle$,
where $\langle -,-\rangle$ is induced by the canonical pairing of $\g$ and $\g^*$.
It descends to a map
\begin{equation*}
    p:
    \Bigl(\Lie(\BI)\big/\bigl(\Lie(\BI)\cap \xx\cdot\Lie(\BI)\bigr)\Bigr)
    \otimes
    \Bigl(\xx((\bb^-)^{\perp}\oplus t\g^*[[t]])\big/
    \bigl(\xx((\bb^-)^{\perp}\oplus t\g^*[[t]])
    \cap((\bb^-)^{\perp}\oplus t\g^*[[t]])\bigr)\Bigr)
    \longrightarrow \C.
\end{equation*}

Let $p_1$ and $p_2$ be the associated maps.
Since the modules are finite-dimensional, it suffices to show that $p_1$ and $p_2$ are injective. Let $s_1\in\Lie(\BI)$ and suppose $p([s_1]\otimes[s_2])=0$ for all $[s_2]$.
Write $\xx^{-1}s_1=\sum_{k\in\Z} a_k t^k$ with $a_k=0$ for $k\ll 0$.
Let $k_0$ be the smallest index such that $a_{k_0}\neq 0$.
If $k_0<0$, choose $s_2=\xx(a' t^{-k_0})$ with $a'\in\g^*$ satisfying
$\langle a_{k_0},a'\rangle\neq 0$.
Then the constant term of
$\langle\xx^{-1}s_1, a't^{-k_0}\rangle$ is nonzero, a contradiction.
Thus $k_0\geq 0$. Similarly, we have $a_0\in\bb^-$.
Hence, $\xx^{-1}s_1\in\Lie(\BI)$ and so $[s_1]=0$. This proves the injectivity of $p_1$. The injectivity of $p_2$ is proved in the same way.
\end{proof}

\begin{lem}\label{lemma_bundle_length_0_and_1}
Let $\xx\in\Waffe$ with $\ell(\xx)=0$ or $1$.
The $\Cxd\times(\BI\rtimes\Cxh)$-equivariant vector bundle
$\widetilde{\BS}_{\leq \xx}$ is isomorphic to $T^*\widetilde{C}_{\leq \xx}$.
\end{lem}

\begin{proof}
If $\ell(\xx)=0$, the assertion follows immediately from \Cref{lemma_fiber_of_BFN_bundle}. 
If $\ell(\xx)=1$, then $C_{\leq\xx}\simeq\PP^1$, so no resolution is needed.
Moreover, $C_{\leq\xx}$ is homogeneous under a minimal parahoric subgroup $\BP_x$ of $G_\ok$.
Since both $\BT_{\leq\xx}$ and $T^*C_{\leq\xx}$ are $\BP_x$-equivariant, and
$\BS|_{\xx}\simeq T^*_{\xx}C_{\leq\xx}$ by Lemma~\ref{lemma_fiber_of_BFN_bundle},
it suffices to check that the last isomorphism is
$\Cxd\times\bigl((\xx\BI\xx^{-1}\cap\BP_x)\rtimes\Cxh\bigr)$-equivariant, for this will induce a surjective vector bundle homomorphism
\begin{equation*}
    \BT_{\leq\xx}\longrightarrow T^*C_{\leq\xx}
\end{equation*}
that restricts to an isomorphism over $C_{\xx}$, yielding the result by \Cref{existresol}. The verification is straightforward and omitted.
\end{proof}

\begin{lem}\label{lemma_b_length_0_and_1}
Let $\xx\in\Waffe$.  
If $\ell(\xx)=0$, then $b_{\xx}=[\xx]$.  
If $\ell(\xx)=1$, then
\begin{equation*}
    b_{\xx}
    =
    \left(\frac{\dd-\xi}{\xi}\right)[\xx]
    -
    \left(\frac{\dd+\xi}{\xi}\right)[\yy]
    =
    k[C_{\leq \xx}] - ([\xx]+[\yy]),
\end{equation*}
where $\yy$ is the unique element $<\xx$, and $\xi$ is the $T\times \Cxh$-weight of $T_{\xx}C_{\leq\xx}$.
\end{lem}
\begin{proof}
    This follows from \Cref{lemma_bundle_length_0_and_1} and localization.
\end{proof}
\begin{lem}\label{lemma_a_is_basis}
    Let $\{a_{\xx}\}_{\xx\in\Waffe}$ be a family of elements of $\sA$. Suppose each $a_{\xx}$ has degree 0 and satisfies $a_{\xx}\equiv [\xx]\pmod{\dd}$. Then $\{a_{\xx}\}_{\xx\in\Waffe}$ is a $\mathbb{Q}[\lt][\hbar,\dd]$-basis of $\sA$.
\end{lem}
\begin{proof}
    Since $\{b_{\yy}\}_{\yy\in\Waffe}$ is a $\mathbb{Q}[\lt][\hbar,\dd]$-basis of $\sA$, we can write $a_{\xx}=\sum_{\yy\in\Waffe}c_{\xx,\yy}b_{\yy}$ for some $c_{\xx,\yy}\in \mathbb{Q}[\lt][\hbar,\dd]$. By the degree condition, each $c_{\xx,\yy}$ belongs to $\mathbb{Q}$. Fix $\xx$, and let $\yy_0\in\Waffe$ be a maximal element for which $c_{\xx,\yy_0}\ne 0$. By \Cref{lemma_fiber_of_BFN_bundle} and the second condition, we have $[\xx]\equiv a_{\xx}\equiv c_{x,y_0}(-1)^{\ell(\yy_0)}[\yy_0]+\cdots\pmod{\dd}$, where $\cdots$ denotes a linear combination of $[\yy]$ with $\yy\ne\yy_0$. It follows that $\yy_0=\xx$. This proves $c_{\xx,\yy}\ne 0$ only if $\yy\leq\xx$, and $c_{\xx,\xx}=(-1)^{\ell(\xx)}\ne 0$. The result now follows.
\end{proof}

\begin{lem}\label{lemma_a_exist_and_unique}
    For any $\xx\in\Waffe$, there exists unique $a_{\xx}\in\sA$ of degree 0 that satisfies $a_{\xx}\equiv [\xx]\pmod{\dd}$.
\end{lem}
\begin{proof}
    If $\ell(\xx)=0$, define $a_{\xx}:=b_{\xx}=[\xx]$. If $\ell(\xx)=1$, define $a_{\xx}:=-b_{\xx}-b_{\yy}$, where $\yy$ is the unique element $<\xx$. By \Cref{lemma_b_length_0_and_1}, they satisfy the required conditions. For arbitrary $\xx$, write $\xx=\xx_1\cdots\xx_{\ell}$, where each $\xx_i$ has length 0 or 1. Define $a_{\xx}:=a_{\xx_1}\cdots a_{\xx_{\ell}}$, which is an element of $\sA$ and has degree 0 because $\sA$ is a graded ring. We have 
    \[ a_{\xx}\equiv [\xx_1]\cdots[\xx_{\ell}]\equiv [\xx]\pmod{\dd}.\]
    This proves the existence.

    It remains to prove the uniqueness. By the existence and \Cref{lemma_a_is_basis}, we have a $\mathbb{Q}[\lt][\hbar,\dd]$-basis $\{a_{\xx}\}_{\xx\in\Waffe}$ of $\sA$ satisfying the two conditions. Suppose $a'_{\xx}$ is another element. Write $a'_{\xx}=\sum_{\yy\in\Waffe}c_{\yy}a_{\yy}$. We have 
    \[ [\xx]\equiv a'_{\xx}\equiv\sum_{\yy\in\Waffe}c_{\yy}[\yy]\pmod{\dd}.\]
    By the degree condition, each $c_{\yy}\in\mathbb{Q}$, and hence $c_{\yy}=\delta_{\xx,\yy}$.
\end{proof}

\begin{prop}[=\Cref{thm_BFN_basis}]\label{maintext_thm_BFN_basis}
    The Iwahori--Coulomb branch algebra $\sA$ admits a $\Q[\lt][\hbar, \dd]$-basis $\{\DD_{\xx}\}_{\xx \in \Waffe}$. Each $\DD_{\xx}$ is uniquely determined by the following properties:
    \begin{enumerate}
        \item $\DD_{\xx}$ has degree $0$; and
        \item $\DD_{\xx} \equiv [\xx] \pmod{\dd}$.
    \end{enumerate}
    Moreover, we have
    \begin{equation*}
        \DD_\xx = (-\dd)^{\ell(\xx)} [C_{\leq\xx}] + \text{terms with smaller $\dd$-degree}.
    \end{equation*}
\end{prop}
\begin{proof}
Define $\DD_{\xx}:=a_{\xx}$, where $a_{\xx}$ is given in \Cref{lemma_a_exist_and_unique}. Then $\{\DD_{\xx}\}_{\xx\in\Waffe}$ is a basis by \Cref{lemma_a_is_basis}. It remains to prove \Cref{eq_BFN_basis}. Note that it follows from \Cref{lemma_b_length_0_and_1} if $\ell(\xx)\leq 1$, and the general case follows from the fact that 
\begin{equation*}
    [C_{\leq \xx}]\cdot [C_{\leq \yy}]=[C_{\leq\xx\yy}]
\end{equation*}
if $\ell(\xx)+\ell(\yy)=\ell(\xx\yy)$.
\end{proof}

From the proof of \Cref{maintext_thm_BFN_basis}, we also obtain an identification of $\sA$ with the trigonometric double affine Hecke algebra, which we now explain.
\begin{df}\label{def_tDAHA}
The \emph{trigonometric double affine Hecke algebra (tDAHA)} $\HH_{G,\hbar,\dd}$ associated with $G$ is the $\Q[\lt][\hbar,\dd]$-algebra generated by elements $L_x$, $x\in \Waffe$, such that $\Q[\hbar,\dd]$ is central, and
\begin{align*}
    L_xL_y &= L_{xy} && x,y\in\Waffe,\\
    L_{s_i}a - s_i(a)L_{s_i} &= -k\langle a,\alpha_i^\vee\rangle && a\in\lt^\vee,\ i=1,2,\dots,r,\\
    L_{s_0}a - s_0(a)L_{s_0} &= k\langle a,\theta^\vee\rangle && a\in\lt^\vee,\\
    L_{\pi}a &= \pi(a)L_\pi && a\in\lt^\vee,\ \pi\in \Waffe,\ \ell(\pi)=0.
\end{align*}
\end{df}

There is a faithful representation of the tDAHA on $\Q[\lt][\hbar,\dd]$, called the difference-rational polynomial representation (\cite{Cherednik_2005}, Section~2.12.3). We recall it as follows. There is a $\Waffe$-action on $\lt\oplus \C_\hbar$ characterized by
\begin{equation*}
    \xx(a,\hbar) = \bigl(u(a),\,\hbar+\langle a,\lambda^\vee\rangle\bigr)
\end{equation*}
for any $\xx = u t_\lambda\in\Waffe$ and $(a,\hbar)\in\lt\oplus\C_\hbar$.  
We write $\theta$ for the highest positive root, and $\theta^\vee\in \Lambda$ for the corresponding coroot.  
We set $s_0 := s_\theta t_{-\theta^\vee}$ and $\alpha_0 := -\hbar-\theta$.  
Then $\Waffe$ is generated by $s_0,s_1,\dots,s_r$ together with the length-$0$ elements in $\Waffe$.

The difference-rational polynomial representation of $\HH_{G,\hbar,\dd}$ is defined in such a way that the elements of $\Q[\lt][\hbar,\dd]$ act by multiplication, and if an element $\xx\in \Waffe$ is written in reduced form as
\begin{equation*}
    \xx = \pi s_{i_1}s_{i_2}\cdots s_{i_\ell},
\end{equation*}
then it acts via
\begin{equation*}
    \pi \circ S_{i_1}\circ S_{i_2}\circ \cdots \circ S_{i_\ell},
\end{equation*}
where the action of $\pi$ is induced from the action on $\lt\oplus\C_\hbar$, and
\begin{equation*}
    S_{i} := s_i + \frac{\dd}{\alpha_i}(s_i-1), \quad i=0,1,...,r.
\end{equation*}

Now we can prove the following result.

\begin{cor}[=\Cref{cor_Coulomb_branch=tDAHA}]\label{maintext_cor_Coulomb_branch=tDAHA}
    The Iwahori--Coulomb branch algebra $\sA=\sA_{G, \mathfrak{g}^*, (\mathfrak{b}^-)^{\perp}}$ is isomorphic to the trigonometric double affine Hecke algebra (tDAHA) $\HH_{G,\hbar,\dd}$.
\end{cor}
\begin{proof}
The algebra
\begin{equation*}
    H^{\hBI\rtimes\Cxh}_\bullet\bigl((\Fl_G)^T\bigr)_\loc
    \simeq
    \Q[\lt][\hbar,\dd]_\loc\rtimes \Waffe
\end{equation*}
naturally acts on $\Q[\lt][\hbar,\dd]_\loc$.
By \Cref{lemma_b_length_0_and_1}, we have $\DD_\pi=[\pi]$ for $\ell(\pi)=0$, and
\begin{equation*}
    \DD_{s_i} = [s_i] + \frac{\dd}{\alpha_i}\bigl([s_i]-1\bigr),\quad i=0,1,\ldots,r.
\end{equation*}
Moreover, \Cref{maintext_thm_BFN_basis}(2) shows that $\sA$ is generated by these elements. This identifies $\sA$ with the image of $\HH_{G,\hbar,\dd}$ under the difference–rational polynomial representation.
\end{proof}

\subsection{\texorpdfstring{Proof of \Cref{Theorem B}}{Proof of Theorem B}}\label{subsection_compute_action}
We begin by introducing some notation. Let $R$ be the set of roots associated with $(G,T)$ and $R^+\subseteq R$ be the set of positive roots associated with $B$. Let $P\supseteq B$ be a standard parabolic subgroup of $G$, $R_P\subseteq R$ be the set of roots of the Levi subgroup of $P$, and $R_P^+ := R_P \cap R^+$. We have the decomposition
\begin{equation*}
    \Lie P = \Lie B \oplus \bigoplus_{\alpha\in R_P^+} \g_{-\alpha}.
\end{equation*}
Write $\PPP:=G/P$ and $X:=T^*\PPP$. Denote the Weyl group of $P$ by $W_P$. Each $u\in W/W_P$ corresponds to a $T$-fixed point $uP\in \PPP\subseteq X$. As in \Cref{localizationsection}, we identify $H^T_2(\pt;\Z)$ with $\Lambda$. For $\lambda\in \Lambda$, we write $\beta_{u,\lambda}$ for the image of $\lambda$ under the homomorphism
\begin{equation*}   H^T_\bullet(\pt;\Z)\longrightarrow H^T_\bullet(X;\Z)\longrightarrow H_\bullet^G(X;\Z),
\end{equation*}
where the first map is induced by the inclusion of the fixed point $uP$. Similarly, we identify
\begin{equation*}
    H^G_2(X;\Z) \simeq H^G_2(\PPP;\Z) \simeq H^P_2(\pt;\Z)
\end{equation*}
with $\Lambda_{W_P}$, the module of $W_P$--coinvariants of $\Lambda$. By abuse of notation, we still denote the image of $\lambda$ in $\Lambda_{W_P}$ by $\lambda.$

\begin{lem}\label{qT to qG}
Under the identification $H_2^G(X;\Z)\simeq \Lambda_{W_P}$, the class $\beta_{u,\lambda}$ is mapped to $u^{-1}(\lambda)$.
\end{lem}

\begin{proof}
This is \cite[Lemma 6.7]{shiftoperatorsCoulombbranches}.
\end{proof}

The natural left $G$-action on $\PPP$ induces a Hamiltonian $G$-action on $X$ whose moment map $f:X\to\g^*$ is proper. Let $\Cxd$ act on $X$ by scaling the cotangent fibers with weight $1$. Then $f$ is $\hG$-equivariant. It is known \cite{NamikawaII} that $f$ factors through a symplectic resolution $X\to\spec H^0(X;\oh_X)$ and a finite morphism $\spec H^0(X;\oh_X)\to\g^*$ with image equal to a nilpotent orbit closure. In particular, the results in \Cref{subsection_symplectic_resolutions} are applicable.

Let $\tau\in \Lambda$ be a generic cocharacter that is positive on all $\alpha\in R^+$. In the case $X=T^*\PPP$, this determines the chamber of $\tau$ uniquely. Let $\epsilon$ be the polarization defined by $\epsilon_{uP}=e^T(T_{uP}\PPP)$.
\begin{df}
    Let $u\in W/W_P$. Define
    \[ \stabp(u):=\stab_{\tau}(uP)\]
    and
    \[ \stabm(u):=\stab_{-\tau}(uP),\]
    where the stable envelopes are defined with respect to the polarization $\epsilon$.
\end{df}

\begin{thm}[=\Cref{Theorem B}]\label{maintext_thm_compute_action}
For any $\xx=wt_\lambda\in\Waffe$ and $u\in W/W_P$, we have
\begin{equation*}
    \DD_\xx\cdot\stabm(u) = (-1)^{d_{u,\lambda}}\, q^{u^{-1}(\lambda)}\, \stabm(wu),
\end{equation*}
where 
\begin{equation*}
    d_{u,\lambda} := \langle u(2\rho_P),\lambda\rangle,
    \qquad 
    2\rho_P := \sum_{\alpha\in R^+\setminus R_P^+} \alpha.
\end{equation*}
\end{thm}

\begin{proof}
Recall $\{\stabm(u)\}_{u\in W/W_P}$ and $\{(-1)^{\dim \PPP}\stabp(v)\}_{v\in W/W_P}$ are dual bases with respect to $(-,-)^{\TT}_X$. It suffices to verify
\begin{equation}\label{eqpairing}
    \bigl(\DD_\xx \cdot \stabm(u),\, (-1)^{\dim\PPP} \stabp(v)\bigr)_X^{\hT}
    = (-1)^{d_{u,\lambda}} q^{u^{-1}(\lambda) } \delta_{wu,v}
\end{equation}
for any $v\in W/W_P$. Since $\deg \DD_\xx = 0$ by \Cref{maintext_thm_BFN_basis}(1), \Cref{maintext_thm_BFN_action_symplectic_resolution} implies that the left-hand side of
\eqref{eqpairing} lies in $\Q[[q_G]]$, and so we may evaluate it by setting $k=0$.

By \Cref{maintext_thm_BFN_basis}(2), we have $\DD_{\xx}\equiv[\xx]\pmod\dd$, and by \Cref{MO_stable_envelope}, we have $\stabpm(v)\equiv [T^*_{vP}\PPP] \pmod\dd$. Therefore,
\begin{align*}
    \bigl(\DD_\xx \cdot \stabm(u),\, (-1)^{\dim\PPP} \stabp(v)\bigr)_X^{\hT}=~&\left.\bigl(\DD_\xx \cdot \stabm(u),\, (-1)^{\dim\PPP} \stabp(v)\bigr)_X^{\hT}\right|_{\dd=0}\\
    =~&\left.\bigl([\xx]\cdot [T^*_{uP}\PPP], (-1)^{\dim\PPP}[T^*_{vP}\PPP]\bigr)_X^{\hT}\right|_{\dd=0}\\
    =~&\left.\bigl(\bbS_{\xx}([T^*_{uP}\PPP]), (-1)^{\dim\PPP}[T^*_{vP}\PPP]\bigr)_X^{\hT}\right|_{\dd=0}\\
    =~ &(-1)^{d_{u,\lambda}} q^{u^{-1}(\lambda) } \delta_{wu,v},
\end{align*}
where the last equality follows from \Cref{lemma_S_x} and \Cref{springercalc} (in this case, the weights $\alpha_{i,1},\ldots,\alpha_{i,k}$ can be taken to be the weights of $T_{uP}\PPP$).
\end{proof}

\subsection{Confluent limit}\label{section_compute_action_G/P}
In this subsection, we apply \Cref{maintext_thm_compute_action} to compute the Iwahori--Coulomb branch action for $X=\PPP$. We take $\bN=\bV=\mathbf{0}$, in which case we have
\begin{equation*}
    \sA_{G,\mathbf{0},\mathbf{0}}^{\Fl}|_{k=0}\simeq \NH_G := H^{T\times\Cxh}_\bullet(\Fl_G).
\end{equation*}
(The notation $\NH_G$ refers to the affine nil-Hecke algebra of $G$. See~\cite{KK86,KacMoodyBook}.)

\begin{df}$~$
    \begin{enumerate}
        \item Let $\xx\in\Waffe$. Define
        \[ D_{\xx}:=[C_{\leq\xx}]\in H^{T\times\Cxh}_{2\ell(\xx)}(\Fl_G).\]

        \item Let $u\in W/W_P$. Define 
        \[ \sigma(u):=\PD\left(\left[\overline{B^-uP/P}\right]\right)\in H_{T\times\Cxh}^{2\ell_P(u)}(\PPP),\]
        where $\ell_P(u)$ is the minimal length of elements representing $uW_P$.
    \end{enumerate}
\end{df}
Note that $\{D_\xx\}_{\xx\in \Waffe}$\footnote{Each $D_\xx$ corresponds to a Demazure element in the affine nil-Hecke algebra.} and $\{\sigma(u)\}_{u\in W/W_P}$ form $\Q[\lt][\hbar]$-bases of $\NH_G$ and $H_{T\times\Cxh}^{\bullet}(\PPP)$, respectively. In what follows, we identify
\begin{equation*}
    H^\bullet_{\hT}(T^*\PPP)\simeq H^\bullet_{\hT}(\PPP)\simeq H^\bullet_{T}(\PPP)[\dd]
\end{equation*}
via the pullback by the projection $T^*\PPP\to \PPP$. By the inductive construction of stable envelopes (see \cite[Section 3.5]{MO}), the following holds:

\begin{lem}\label{lemma_stable_envelope_limit}
    For any $u\in W/W_P$, we have
    \[\stabm(u)
    = (-\dd)^{\dim\PPP-\ell_P(u)}\, \sigma(u)
    + \text{terms with smaller $\dd$-degree}.\tag*{\qed}\]
\end{lem}

\begin{df}\label{df_Pallowed}
Let $\xx = w t_\lambda \in \Waffe$ and $u\in W/W_P$.
We say that $(\xx,u)$ is \emph{$P$-allowed} if, for every $\alpha\in R^+\setminus R_P^+$,
\begin{equation*}
\begin{cases}
    \langle u(\alpha),\lambda\rangle \le 0
        & \text{if $u(\alpha)$ and $w u(\alpha)$ are both positive or both negative},\\[4pt]
    \langle u(\alpha),\lambda\rangle \le 1
        & \text{if $u(\alpha)<0$ and $w u(\alpha)>0$},\\[4pt]
    \langle u(\alpha),\lambda\rangle \leq -1
        & \text{if $u(\alpha)>0$ and $w u(\alpha)<0$,}
\end{cases}
\end{equation*}
and, for every $\alpha\in R^+\cap u(R_P)$,
\begin{equation*}
\begin{cases}
    \langle \alpha,\lambda\rangle = 0
        & \text{if $w(\alpha)>0$,}\\[4pt]
    \langle \alpha,\lambda\rangle = -1
        & \text{if $w(\alpha)<0$.}
\end{cases}
\end{equation*}
\end{df}

For $\lambda\in \Lambda$, we write $\lambda^-$ for the unique antidominant element in the orbit $W\cdot \lambda$.
We write $w_\lambda^-$ for the minimal length element in the coset
$w_\lambda^- Z_W(\lambda^-)$. Equivalently, $w_\lambda^- t_{\lambda^-}$ is the shortest element in $\Waffe$ mapping to $\lambda\in \Lambda\simeq \Waffe/W$.

\begin{df}
We say that $\lambda\in \Lambda$ is \emph{$P$-allowed} if, for any $\alpha\in R_P^+$,
\begin{equation*}
\langle \alpha, \lambda^- \rangle =
\begin{cases}
    0 & \text{if $w^-_\lambda(\alpha)>0$,}\\
    -1 & \text{if $w^-_\lambda(\alpha)<0$.}
\end{cases}
\end{equation*}
\end{df}

The following lemma is immediate.
\begin{lem}\label{Lem_Pallowed}
Let $\lambda\in\Lambda$. Then $(w_\lambda^- t_{\lambda^-},e)$ is $P$-allowed if and only if $\lambda$ is $P$-allowed.
\end{lem}
The following is the main result of this subsection.

\begin{thm}[=\Cref{thm_compute_action_G/P}]\label{maintext_thm_compute_action_G/P}
For any $\xx = w t_\lambda \in \Waffe$ and $u\in W/W_P$, we have
\begin{equation*}
    D_{\xx}\cdot \sigma(u)
    =
    \begin{cases}
        q^{u^{-1}(\lambda) }\,\sigma(w u)
            & \text{if $(x,u)$ is $P$-allowed}\\[4pt]
        0 & \text{otherwise.}
    \end{cases}
\end{equation*}
\end{thm}

\begin{cor}[=\Cref{thm_new_proof_of_PLS}]\label{maintext_thm_new_proof_of_PLS}
There is a graded ring homomorphism $\Upsilon_P:H^T_\bullet(\Gr_G)\to H_T^\bullet(\PPP)[q_G]$ satisfying
\begin{equation*}
    \Upsilon_P([C_{\leq\lambda}]) =
\begin{cases}
    q^{\lambda^-}\sigma(w_\lambda^-)
        & \text{if $\lambda$ is $P$-allowed}\\
    0 & \text{otherwise.}
\end{cases}
\end{equation*}
In particular, the homomorphism
\begin{equation*}
    \Upsilon_B:H^{T}_\bullet(\Gr_G)\to H_T^\bullet(\BBB)[q_G]
\end{equation*}
becomes an isomorphism after localizing $H^T_\bullet(\Gr_G)$ by the classes $[C_{\leq\lambda}]$ for all dominant $\lambda$.
\end{cor}
\begin{proof}[Proof of \Cref{maintext_thm_new_proof_of_PLS}]
Let $\xx = w_\lambda^- t_{\lambda^-}$. Note that the natural map $C_\xx \to C_\lambda$ is an isomorphism, and hence by \Cref{Results_of_CCL}(2)
\begin{equation*}
    \Psi_{\PPP}^{\hbar=0}([C_{\leq \lambda}])
    =
    [C_{\leq \lambda}]\cdot 1
    =
    D_\xx\cdot 1=\Upsilon_P([C_{\leq\lambda}]).
\end{equation*}
Here, the last equality follows from \Cref{maintext_thm_compute_action_G/P} together with the fact that $\sigma(e)=1$. By \Cref{Results_of_CCL}(3), this defines a graded ring homomorphism.
\end{proof}

Following \cite{BMO}, we introduce an automorphism $\kappa$ of $H_{\TT}^\bullet(T^*\PPP)[q_G]$ that is the identity on $H_{\TT}^\bullet(T^*\PPP)$ and acts on Novikov parameters by
\begin{equation*}
   \kappa:\ q^\beta\longmapsto \dd^{-\langle c_1^G(\PPP),\beta\rangle}\, q^{\beta}.
\end{equation*}

\begin{lem}\label{lemchangeNovi}
\begin{enumerate}
    \item Let $\phi_i\in H_{T}^\bullet(\PPP)$ and $\beta\in H_2(\PPP)$.
    Then $\langle \phi_1,\dots,\phi_{n-1}, e(T^*\PPP)\phi_n\rangle^{T^*\PPP}_{0,n,\beta}$ is defined without localization, and
    \begin{equation*}
        \lim_{k\to\infty} \dd^{-\langle c_1^G(\PPP),\beta\rangle}
        \langle \phi_1,\dots,\phi_{n-1}, e(T^*\PPP)\phi_n\rangle^{T^*\PPP}_{0,n,\beta}
        =
        \langle \phi_1,\dots,\phi_n\rangle^{\PPP}_{0,n,\beta}.
    \end{equation*}

    \item For $\gamma_1,\gamma_2\in H_{T}^\bullet(\PPP)$, we have
    \begin{equation*}
        \lim_{k\to\infty}\kappa(\gamma_1\star^{T^*\PPP} \gamma_2)
        =
        \gamma_1\star^{\PPP}\gamma_2.
    \end{equation*}

    \item For $\Gamma\in H_\bullet^{\BI}(\Fl_G)$, we have
    \begin{equation*}
        \lim_{k\to\infty}\kappa\circ \bbS_{T^*\PPP}(\Gamma,-)
        =
        \bbS_{\PPP}(\Gamma,-).
    \end{equation*}
\end{enumerate}
\end{lem}

\begin{proof}[Proof of \Cref{thm_compute_action_G/P}]
By \Cref{maintext_thm_compute_action}, we have
\begin{equation*}\label{eq_before_lim}
    \bbS_{T^*\PPP}(\DD_{\xx}, \stabm(u)) = \DD_{\xx}\cdot\stabm(u)
    = (-1)^{d_{u,\lambda}}\, q^{u^{-1}(\lambda)}\, \stabm(wu).
\end{equation*}
By \Cref{maintext_thm_BFN_basis} and \Cref{lemma_stable_envelope_limit}, we have
\[\bbS_{T^*\PPP}(\DD_{\xx}, \stabm(u))=(-\dd)^{\ell(\xx)+\dim\PPP-\ell_P(u)}\bbS_{T^*\PPP}(D_{\xx},\sigma(u))+\sum_{i=0}^{\ell(\xx)+\dim\PPP-\ell_P(u)-1}\dd^i\bbS_{T^*\PPP}(\Gamma_i,\gamma_i) \]
for some $\Gamma_i\in H_{\bullet}^{T\times\Cxh}(\Fl_G)$ and $\gamma_i\in H_T^{\bullet}(\PPP)$. Thus, by \Cref{lemchangeNovi}, applying $\lim_{\dd\to\infty} k^{-\ell(\xx)-\dim\PPP+\ell_P(u)}\kappa$ to both sides gives
\begin{align*}
    & (-1)^{\ell(\xx)+\dim \PPP-\ell_P(u)}\bbS_{\PPP}(D_{\xx},\sigma(u))\\
    =~& \lim_{\dd\to\infty} k^{-\ell(\xx)-\dim\PPP+\ell_P(u)}\kappa\left((-1)^{d_{u,\lambda}}\, q^{u^{-1}(\lambda)}\, \stabm(wu)\right)\\
    =~&\lim_{\dd\to\infty} k^{-\ell(\xx)+\ell_P(u)-\langle c_1^G(\PPP),\beta_{u,\lambda}\rangle-\ell_P(wu)}q^{u^{-1}(\lambda)}\left((-1)^{d_{u,\lambda}+\dim\PPP-\ell_P(wu)}\sigma(wu)+F(k)\right),
\end{align*}
where $F(k)$ is a polynomial in $k^{-1}$ without constant term.

By \Cref{lemma_deg_nonnegative} below and using $\langle c_1^G(\PPP),\beta_{u,\lambda}\rangle=-d_{u,\lambda}=-\sum_{\alpha\in R^+\setminus R_P^+}\langle \alpha,u^{-1}(\lambda)\rangle$, the exponent $-\ell(\xx)+\ell_P(u)-\langle c_1^G(\PPP),\beta_{u,\lambda}\rangle-\ell_P(wu)$ is nonpositive and is zero if and only if $(\xx,u)$ is $P$-allowed, in which case the signs $(-1)^{\ell(\xx)+\dim\PPP-\ell_P(u)}$ and $(-1)^{d_{u,\lambda}+\dim\PPP-\ell_P(wu)}$ from each side are equal. The result follows.
\end{proof}

\begin{proof}[Proof of \Cref{lemchangeNovi}]\ 
\begin{enumerate}
\item
The first assertion follows because the evaluation morphisms are proper, and $e(T^*\PPP)$ is Poincar\'e dual to the zero section, which is complete. We sketch the localization argument, which is similar to Proposition~8.1 in \cite{BMO}.

Let $\sigma:(\Sigma,p_1,\dots,p_n)\to T^*\PPP$ be a fixed point in the moduli space.
Let $\mathcal{N}_\sigma^{T^*\PPP}$ and $\mathcal{N}_\sigma^{\PPP}$ be the virtual tangent spaces at $\sigma$ in the moduli spaces of stable maps to $T^*\PPP$ and to $\PPP$, respectively. Then
\begin{equation*}
    \mathcal{N}_\sigma^{T^*\PPP}
    =
    \mathcal{N}_\sigma^{\PPP}
    -
    R\Gamma^*(\Sigma,\sigma^*T^*\PPP).
\end{equation*}
On the other hand, there is a short exact sequence
\begin{equation*}
    0\to E\to \sigma^*T^*\PPP\to T^*\PPP|_{\sigma(p_n)}\to 0.
\end{equation*}
Therefore,
\begin{equation*}
    \mathcal{N}_\sigma^{T^*\PPP}
    =
    \mathcal{N}_\sigma^{\PPP}
    -
    T^*_{\sigma(p_n)}(\PPP)
    +
    H^1(\Sigma,E).
\end{equation*}
Moreover, $H^1(\Sigma,E)$ contributes a factor of $\dd^{\langle c_1(\PPP),\beta\rangle}$ to the localization weight, up to terms of lower order in $\dd$.
This yields the claimed limit by the virtual localization theorem (see~\cite{vloc}).

\item
Let $\{\phi_i\}$ and $\{\phi^i\}$ be dual $H_T^\bullet(\pt)$-bases of $H_T^\bullet(\PPP)$.
Note that $\{e(T^*\PPP)\phi_i\}$ and $\{\phi^i\}$ are dual bases in $H_{\TT}^\bullet(T^*\PPP)$.

We compute
\begin{align*}
    \lim_{k\to\infty}\kappa(\gamma_1\star^{T^*\PPP}\gamma_2)
    &=
    \lim_{k\to\infty}
    \sum_{i,\beta}\kappa(q^\beta)\,
    \langle \gamma_1,\gamma_2,e^{G\times\Cx}(T^*\PPP)\phi^i\rangle^{T^*\PPP}_{0,3,\beta}\,\phi_i \\
    &=
    \sum_{i,\beta}
    q^\beta\,
    \langle \gamma_1,\gamma_2,\phi^i\rangle^{\PPP}_{0,3,\beta}\,\phi_i \\
    &=
    \gamma_1\star^{\PPP}\gamma_2,
\end{align*}
where we used part~(1) in the second equality.

\item
Similarly to part~(2), we have
\begin{equation*}
    \lim_{k\to\infty}\kappa(\mathbb{M}^{T^*\PPP})=\mathbb{M}^{\PPP}.
\end{equation*}
Therefore, by \Cref{localizationthm}, it suffices to check that for each $u\in W/W_P$ and $\lambda\in\Lambda$,
\begin{align*}
&\lim_{k\to\infty}\kappa(q^{\beta_{u,\lambda}})
\prod_{\alpha\in u(R^-\setminus R_P^-)}
\frac{\prod_{c=0}^{\infty}(\alpha+c\hbar)}
     {\prod_{c=\alpha(\lambda)}^{\infty}(\alpha+c\hbar)}
\frac{\prod_{c=0}^{\infty}(-\alpha+\dd+c\hbar)}
     {\prod_{c=-\alpha(\lambda)}^{\infty}(-\alpha+\dd+c\hbar)}
\\
&\qquad\qquad
=
q^{\beta_{u,\lambda}}
\prod_{\alpha\in u(R^-\setminus R_P^-)}
\frac{\prod_{c=0}^{\infty}(\alpha+c\hbar)}
     {\prod_{c=\alpha(\lambda)}^{\infty}(\alpha+c\hbar)}.
\end{align*}
This follows from the identity
\begin{equation*}
    \langle c_1(T^*\PPP),\beta_{u,\lambda}\rangle
    =
    \sum_{\alpha\in u(R^+\setminus R_P^+)}\langle \alpha,\lambda\rangle.\qedhere
\end{equation*}
\end{enumerate}
\end{proof}

\begin{lem}\label{lemma_deg_nonnegative}
    Let $\xx=wt_{\lambda}\in\Waffe$ and $u\in W$. We have
    \[\ell(\xx)-\ell_P(u)+\ell_P(wu)+\langle 2\rho_P,u^{-1}(\lambda)\rangle\geq 0.\]
    The equality holds if and only if $(\xx,u)$ is $P$-allowed.
\end{lem}
\begin{proof}
    Define
\begin{align*}
    S_0 &:= R^+ \setminus R_P^+\\
    S_1 &:= \{\alpha \in R^+ \cap w^{-1}(R^-)\mid \langle \alpha,\lambda\rangle \ge 0\}\\
    S_2 &:= \{\alpha \in R^+ \cap w^{-1}(R^-)\mid \langle \alpha,\lambda\rangle < 0\}\\
    S_3 &:= (R^+ \setminus R_P^+) \cap u^{-1}(R^-)\\
    S_4 &:= (R^+ \setminus R_P^+) \cap (wu)^{-1}(R^-).
\end{align*}
We have 
\begin{align*}
    \ell(\xx) &=\sum_{\alpha\in R^+}|\langle\alpha,\lambda\rangle|+\#S_1-\#S_2\\
    \ell_P(u)&=\#S_3\\
    \ell_P(wu)&=\#S_4\\
    \langle 2\rho_P,u^{-1}(\lambda)\rangle&=\sum_{\alpha\in S_0}\langle u(\alpha),\lambda\rangle.
\end{align*}

It follows that
\begin{align}\label{comparison3}
    & \ell(\xx)-\ell_P(u)+\ell_P(wu)+\langle 2\rho_P,u^{-1}(\lambda)\rangle \nonumber\\
    =~& \sum_{\alpha\in R^+}
    \Bigl(
        \bigl|\langle \alpha,\lambda\rangle\bigr|
        + \chi_{S_0}(\alpha)\langle u(\alpha),\lambda\rangle
        + \chi_{S_1}(\alpha) - \chi_{S_2}(\alpha)
        - \chi_{S_3}(\alpha) + \chi_{S_4}(\alpha)
    \Bigr)\nonumber \\
    =~& \sum_{\alpha\in R^+}
    \Bigl(
        \bigl|\langle u(\alpha),\lambda\rangle\bigr|
        + \chi_{S_0}(\alpha)\langle u(\alpha),\lambda\rangle
        + \chi_{S_1}(\alpha') - \chi_{S_2}(\alpha')
        - \chi_{S_3}(\alpha) + \chi_{S_4}(\alpha)
    \Bigr), 
\end{align}
where $\chi_{S_i}(\alpha)=1$ if $\alpha\in S_i$ and $0$ otherwise, and $\alpha'$ is the unique positive root in $\{\pm u(\alpha)\}$.

Let $g(\alpha)$ be the summand in the last sum in \Cref{comparison3}.
\begin{enumerate}
    \item[Case 1a:] If $\alpha\not \in R_P^+$ with $u(\alpha)$ and $wu(\alpha)$ both positive or both negative, then
    $g(\alpha)= \bigl|\langle u(\alpha),\lambda\rangle\bigr|
        + \langle u(\alpha),\lambda\rangle\geq 0$.
    \item[Case 1b:] If $\alpha\not \in R_P^+$, $u(\alpha)<0$, $wu(\alpha)>0$ and $\langle u(\alpha),\lambda\rangle\le 0$, then $g(\alpha)=0$.
    \item[Case 1c:] If $\alpha\not \in R_P^+$, $u(\alpha)<0$, $wu(\alpha)>0$ and $\langle u(\alpha),\lambda\rangle> 0$, then $g(\alpha)= 2\langle u(\alpha),\lambda\rangle-2\geq 0$.
    \item[Case 1d:] If $\alpha\not \in R_P^+$, $u(\alpha)>0$, $wu(\alpha)<0$ and $\langle u(\alpha),\lambda\rangle\geq 0$, then $g(\alpha)= 2\langle u(\alpha),\lambda\rangle+2> 0$.
    \item[Case 1e:] If $\alpha\not \in R_P^+$, $u(\alpha)>0$, $wu(\alpha)<0$ and $\langle u(\alpha),\lambda\rangle< 0$, then $g(\alpha)= 0$.
    \item[Case 2a:] If $\alpha\in R_P^+$ and $w(\alpha')>0$, then $g(\alpha)=\bigl|\langle \alpha',\lambda\rangle\bigr|\geq 0$.
    \item[Case 2b:] If $\alpha\in R_P^+$, $w(\alpha')<0$ and $\langle \alpha',\lambda\rangle\geq 0$, then $g(\alpha)=\langle \alpha',\lambda\rangle+1 > 0$.
    \item[Case 2c:] If $\alpha\in R_P^+$, $w(\alpha')<0$ and $\langle \alpha',\lambda\rangle< 0$, then $g(\alpha)=-\langle \alpha',\lambda\rangle-1 \ge 0$.
\end{enumerate}
Therefore, each $g(\alpha)\geq 0$, and the equality holds for all $\alpha\in R^+$ if and only if $(\xx,u)$ is $P$-allowed.
\end{proof}

\subsection{Namikawa--Weyl group action}\label{subsection_quantum_Weyl_group_action}
Let $N_W(W_P)$ be the normalizer of $W_P$ in $W$, and set
\begin{equation*}
    \qw_P := N_W(W_P)/W_P.
\end{equation*}
By \cite{NamiNilpII} (see also \cite{BLPW}), $\qw_P$ contains the Namikawa--Weyl group (or the symplectic Galois group) of $T^*\PPP$. 
\begin{df}\label{df_Namikawa_Weyl_group_action}
Define an $H_{\TT}^{\bullet}(\pt)_{\loc}$-linear action $\qwa$ of $\qw_P$ on $H_{\TT}^{\bullet}(T^*\PPP)_{\loc}[q_G]$ by
\[ w\qwa\left(q^{\lambda}\stabm(u)\right):=(-1)^{\ell_P(w)}q^{w(\lambda)}\stabm(uw^{-1})\] 
for $w\in\qw_P$, $\lambda\in \Lambda_P$ and $u\in W/W_P$.
\end{df}

\begin{prop}\label{Namikawa_commutes_A}
The $\qw_P$-action $\qwa$ commutes with the Iwahori--Coulomb branch action of $\sA$.
\end{prop}
\begin{proof}
Since $\qwa$ commutes with the multiplication by equivariant parameters, it suffices to show that
\begin{equation*}
    w\qwa(\DD_\xx\cdot \stabm(v))
    = \DD_\xx\cdot\bigl(w\qwa\stabm(v)\bigr)
\end{equation*}
for any $\xx = u t_\lambda\in \Waffe$, $w\in \qw_P$ and $v\in W/W_P$. By \Cref{maintext_thm_compute_action}, this amounts to verifying
\begin{equation*}
    (-1)^{\ell_P(w)+d_{v,\lambda}}\, q^{w(v^{-1}(\lambda))}\, \stabm(uvw^{-1})
    =
    (-1)^{\ell_P(w)+d_{vw^{-1},\lambda}}\, q^{(vw^{-1})^{-1}(\lambda)}\, \stabm(uvw^{-1}).
\end{equation*}
The only point to check is the sign. Since $w$ preserves $R\setminus R_P$, we have
\begin{equation*}
    d_{v,\lambda}
    = \sum_{\alpha\in R^+\setminus R_P^+}\langle v(\alpha),\lambda \rangle
    \equiv \sum_{\alpha\in R^+\setminus R_P^+}\langle vw^{-1}(\alpha),\lambda \rangle
    \equiv d_{vw^{-1},\lambda}
    \pmod{2}.
\end{equation*}
This proves the claim.
\end{proof}

Before proceeding, let us recall the following well-known result. See e.g., \cite{GKMforhomogeneous}.
\begin{lem}\label{Cohomology_determined_by_restriction}
There is a commutative diagram
\begin{equation*}
    \begin{tikzcd}
        H_G^\bullet(\PPP) \ar[r,phantom, "\subset"] \ar[d, "\simeq"]
            & H_T^\bullet(\PPP) \ar[r, "\operatorname{Res}_{eP}"]
                & H_T^\bullet(\pt) \ar[d, "\simeq"] \\
        \Q[\lt]^{W_P} \ar[rr,phantom,"\subset"]
            &
            & \Q[\lt]
    \end{tikzcd},
\end{equation*}
where $\operatorname{Res}_{eP}$ is the restriction map to the fixed point $eP$.
In particular, an element of $H_T^\bullet(\PPP)_\loc$ lies in $H_G^\bullet(\PPP)$ if and only if it is $W$-invariant and its restriction to $eP$ lies in $\Q[\lt]$.\hfill$\square$
\end{lem}

\begin{lem}\label{qw=A_w}
    We have $\operatorname{Res}_{eB}(w\qwa \gamma)=\DD_w(\operatorname{Res}_{eB}( \gamma))$ for any $w\in \qw_B$ and $\gamma\in H^\bullet_{\hG}(T^*\BBB)$. In particular, the $\qw_B$-action $\qwa$ preserves the submodule $H^\bullet_{\hG}(T^*\BBB)$. 
\end{lem}
\begin{proof}
We may assume $w=s$ is the reflection corresponding to a simple root $\alpha$. Set $\Delta:=\stabm(e)|_{eB}=(-1)^{\dim\BBB}\prod_{\alpha'\in R^+}(\dd+\alpha')$ and write $\gamma=\Delta^{-1}\sum_{u\in W}(-1)^{\ell(u)}\gamma_u\stabm(u)$ with $\gamma_u\in\Q(\lt)(\dd)$. Since $s\qwa\stabm(u)=-\stabm(us)$, the first assertion is equivalent to $\gamma_s=\DD_s(\gamma_e)$, or equivalently,
\begin{equation}\label{eq_qw=A_w}
    s(\gamma_e)=\frac{\dd}{\dd+\alpha}\gamma_e+\frac{\alpha}{\dd+\alpha}\gamma_s.
\end{equation}
To show this, we restrict the equality $\stabm(s)=\DD_s\cdot \stabm(e)$ (\Cref{maintext_thm_compute_action}) to $eB$ and use $\DD_s=[s]+\frac{\dd}{\alpha}([s]-1)$ (\Cref{lemma_b_length_0_and_1}) to obtain
\[
\stabm(e)|_{sB}=\frac{\dd}{\dd+\alpha}\Delta.
\]
\Cref{eq_qw=A_w} then follows from expanding the equality $\gamma|_{sB}=s(\gamma|_{eB})$ (since $\gamma$ is $W$-invariant).

Finally, since $s\qwa \gamma$ is $W$-invariant by \Cref{Namikawa_commutes_A}, the last assertion follows from \Cref{Cohomology_determined_by_restriction}.
\end{proof}

We extend \Cref{qw=A_w} to the case of arbitrary $P$ as follows. Let $\pi:\BBB\to \PPP$ be the projection. Consider the induced map
\[ \pi_*: H_{\TT}^{\bullet}(T^*\BBB)\simeq H_{\TT}^{\bullet}(\BBB)\to H_{\TT}^{\bullet}(\PPP)\simeq H_{\TT}^{\bullet}(T^*\PPP).\]

The next two lemmas can be found in \cite{Surestriction,Suthesis}; we give the proofs here for completeness.
\begin{lem}\label{pushforward_stable_basis}
    For any $u \in W$, we have
    \begin{equation*}
        \pi_*(\stabm_{T^*\BBB}(u)) =  \stabm_{T^*\PPP}(u).
    \end{equation*}
\end{lem}
\begin{proof}
    Since $\pi$ is $G$-equivariant, $\pi_*$ is $W$-equivariant and so commutes with $\DD_w$ for all $w\in W$. We have
    \begin{align*}
        \stabm_{T^*\BBB}(u)&=\DD_{uw_0}\cdot\stabm_{T^*\BBB}(w_0)=\DD_{uw_0}\cdot[T^*_{w_0B}\BBB]\\
        \stabm_{T^*\PPP}(u)&=\DD_{uw_0}\cdot\stabm_{T^*\PPP}(w_0)=\DD_{uw_0}\cdot[T^*_{w_0P}\PPP]
    \end{align*}
    by \Cref{maintext_thm_compute_action}. Since $\pi_*[T^*_{w_0B}\BBB]=[T^*_{w_0P}\PPP]$,
    \[\pi_*\left(\stabm_{T^*\BBB}(u)\right)=\pi_*\left(\DD_{uw_0}\cdot [T^*_{w_0B}\BBB]\right) = \DD_{uw_0}\cdot\left(\pi_*[T^*_{w_0B}\BBB]\right)=\DD_{uw_0}\cdot[T^*_{w_0P}\PPP]=\stabm_{T^*\PPP}(u). \qedhere\]
\end{proof}

\begin{lem}\label{pushforward_commutes_with_qw}
    Let $w\in \qw_P$, and $\widetilde w\in W$ be its minimal length representative. We have
    \begin{equation*}
        \pi_*( \widetilde w\qwa-)=w\qwa \pi_*(-).
    \end{equation*}
\end{lem}
\begin{proof}
    This follows from \Cref{pushforward_stable_basis}.
\end{proof}

\begin{prop}\label{lem_qw_preserves_nonlocalized}
    The $\qw_P$-action $\qwa$ preserves the submodule $H^\bullet_{\hG}(T^*\PPP)$.
\end{prop}
\begin{proof}
    Note that $\pi_*: H^\bullet_{\hG}(T^*\BBB) \to H^\bullet_{\hG}(T^*\PPP)$ is surjective. Hence, by \Cref{qw=A_w} and \Cref{pushforward_commutes_with_qw}, we have
    \begin{equation*}
        \qw_P(H^\bullet_{\hG}(T^*\PPP)) = \qw_P(\pi_*(H^\bullet_{\hG}(T^*\BBB))) \subseteq \pi_*(\qw_B(H^\bullet_{\hG}(T^*\BBB))) = \pi_*(H^\bullet_{\hG}(T^*\BBB)) = H^\bullet_{\hG}(T^*\PPP). \qedhere
    \end{equation*}
\end{proof}

\begin{lem}\label{lem_qw_fixes_1}
    The $\qw_P$-action $\qwa$ fixes the element $1\in H^\bullet_{\GG}(T^*\PPP)$.
\end{lem}
\begin{proof}
    Let $w\in\qw_P$. By \Cref{lem_qw_preserves_nonlocalized}, we have $w \qwa 1 \in H^0_{\hG}(T^*\PPP) \simeq \mathbb{Q}$. The equality $w \qwa 1 = 1$ is then verified by setting $\dd=0$. Details are left to the reader.
\end{proof}

\begin{lem}\label{imPsi_is_qw_invariant}
    $\sA\cdot 1$ is $\qw_P$-invariant.
\end{lem}
\begin{proof}
    This follows from \Cref{Namikawa_commutes_A} and \Cref{lem_qw_fixes_1}.
\end{proof}

\begin{lem}\label{lemma_A_span_QH}
$H_{\hG}^\bullet(T^*\PPP)$ is contained in the $\Q(q_G)$-vector space spanned by $\sAs\cdot 1$.
\end{lem}
\begin{proof}
It suffices to prove the statement after setting $\hbar=\dd=0$. Write $1=\sum_v b_v$ according to the decomposition
\begin{equation*}
    H^\bullet_T(T^*\PPP)_\loc \simeq \bigoplus_{v\in W/W_P} H^\bullet_T(vP)_\loc.
\end{equation*}
Let $\lambda\in\Lambda$. We have, by \Cref{maintext_thm_compute_action},
\begin{equation*}
    z_\lambda := [t_\lambda]\cdot 1
    = \sum_{v\in W/W_P} (-1)^{d_{v,\lambda}}\, q^{v^{-1}(\lambda)}\, b_v\in H_T^{\bullet}(T^*\PPP)_{\loc}[q_G].
\end{equation*}

Write $W/W_P=\{v_1,\ldots,v_n\}$ with $n:=|W/W_P|$. Observe that the matrix $(q^{v_k^{-1}(\ell\lambda)})_{1\leq k,\ell\leq n}$ is a Vandermonde matrix, and hence it is nonsingular if $q^{v_1^{-1}(\lambda)},\ldots, q^{v_n^{-1}(\lambda)}$ are pairwise distinct. We now assume $\lambda$ is generic so that the last condition holds. Since the determinants of $((-1)^{d_{v,\lambda}}q^{v_k^{-1}(\ell\lambda)})_{1\leq k,\ell\leq n}$ and $(q^{v_k^{-1}(\ell\lambda)})_{1\leq k,\ell\leq n}$ are equal up to sign, it follows that there exist $c_{k,\ell}\in\Q(q_G)$ ($1\leq k,\ell\leq n$) such that $b_{v_k}=\sum_{\ell=1}^n c_{k,\ell}z_{\ell\lambda}$ for all $k$.

Let $\gamma\in H_G^{\bullet}(T^*\PPP)$. Write $\gamma=\sum_{k=1}^n\gamma_k b_{v_k}$. Note that each $\gamma_k$ belongs to $\Q[\lt]$. We have
\[\gamma=\frac{1}{|W|}\sum_{w\in W}w(\gamma)=\frac{1}{|W|}\sum_{k,\ell=1}^n c_{k,\ell}\sum_{w\in W}w(\gamma_k z_{\ell\lambda})= \frac{1}{|W|}\sum_{k,\ell=1}^n c_{k,\ell}\sum_{w\in W}w(\gamma_k[t_{\ell\lambda}])\cdot 1.\]
Each $\sum_{w\in W}w(\gamma_k[t_{\ell\lambda}])$ is a $W$-invariant element of $\sum_{\mu\in\Lambda}\Q[\lt]\cdot [t_{\mu}]$ and thus an element of $\sAs_{\hbar=\dd=0}$ by \cite[6(vi)]{BFN}. The result follows.
\end{proof}

Recall from \Cref{dfseidelmap} the $\Q[\lt][\hh,\dd]$-linear map
\[
\begin{array}{ccccc}
  \Psi_{T^*\PPP}&:&H_{\bullet}^{\TT\times\Cxh}(\Gr_G) &\longrightarrow &H_{\TT\times\Cxh}^\bullet(T^*\PPP)_{\loc}[[q_G]]\\ [.9em]
&&a&\longmapsto&\bbS_{T^*\PPP}^{\Gr}(a,1).
\end{array}
\]

\begin{prop}[=\Cref{thm_image_shift=qwinvariants}]\label{maintext_thm_image_shift=qwinvariants} The following statements are true:
\begin{enumerate}
    \item $\Psi_{T^*\PPP}(\sAGr)$ is contained in $\bigl(H_{\GG\times\Cxh}^\bullet(T^*\PPP)[q_G]\bigr)^{\qw_P}$.
    \item The $\Q(q_G)$-span of $\Psi_{T^*\PPP}(\sAGr)$ contains $H_{\GG\times\Cxh}^{\bullet}(T^*\PPP)[q_G]$.
    \item $\Psi_{T^*\BBB}$ restricts to a $\Q[\lt]^W[\hbar, \dd]$-module isomorphism 
    \[
    \sAGr\simeq \bigl(H_{\GG\times\Cxh}^\bullet(T^*\BBB)[q_G]\bigr)^{\qw_B}.
    \]
\end{enumerate}
\end{prop}
\begin{proof}
The first and second statements follow from \Cref{imPsi_is_qw_invariant} and \Cref{lemma_A_span_QH}, respectively. To prove the last statement, it suffices to verify that the specialization of $\Psi_{T^*\BBB}$ at $\hh=\dd=0$ restricts to an isomorphism $\sAGr_{\hbar=k=0}\simeq (H_G^\bullet(T^*\BBB)[q_G])^{\qw_B}$.

Note that $\sAs_{\hbar=\dd=0}=\left(\sum_{\lambda\in\Lambda}\Q[\lt]\cdot[t_{\lambda}]\right)^W$ (see \cite[6(vi)]{BFN}). By the formula in \Cref{springercalc}, we have $\left(f[t_{\lambda}]\cdot 1\right)|_{eB}=\pm f q^{\lambda}$. It follows that the composition
\[\sAs_{\hbar=\dd=0}\longrightarrow \left(H_G^{\bullet}(T^*\BBB)[q_G]\right)^{\qw_B}\xrightarrow{\operatorname{Res}_{eB}}\left(H_T^{\bullet}(eB)[q_G]\right)^W\]
is bijective. The result follows from \Cref{Cohomology_determined_by_restriction}.
\end{proof}

The following corollary generalizes a result in \cite{LSX} which assumes $P=B$.
\begin{cor}[=\Cref{thm_LSX}]\label{maintext_thm_LSX}
The submodule $H_{\GG}^{\bullet}(T^*\PPP)(q_G)$ is closed under the quantum product $\star$. Furthermore, $\star$ is preserved by the $\qw_P$-action, that is,
    \begin{equation*}
        w\qwa(\gamma_1\star \gamma_2)= (w\qwa \gamma_1)\star(w\qwa \gamma_2)
    \end{equation*}
    for all $w\in \qw_P$ and $\gamma_1,\gamma_2\in H_{\GG}^{\bullet}(T^*\PPP)(q_G)$.
\end{cor}
\begin{proof}
    By \Cref{Results_of_CCL}(3), $\Psi^{\hbar=0}_{T^*\PPP}$ is a ring homomorphism. The first assertion then follows from \Cref{maintext_thm_image_shift=qwinvariants}.
    
    For the second assertion, let $\gamma, \gamma' \in H_{\hG}^\bullet(T^*\PPP)(q_G)$ and $w \in \qw_P$.
    By \Cref{maintext_thm_image_shift=qwinvariants}(2), we may write $\gamma=\sum c_ia_i$, $\gamma'=\sum c'_ja'_j$, where $c_i,c'_j\in \Q(q_G)$, and $a_j,a'_j\in\operatorname{Im}\left(\left.\Psi^{\hbar=0}_{T^*\PPP}\right|_{\scriptstyle{\sAGr_{\hh=0}}}\right)$. We have $\gamma\star \gamma'=\sum c_ic'_ja_i\star a'_j$, and
    \begin{equation*}
        (w\qwa\gamma)\star (w\qwa\gamma')=\sum (w\qwa c_i)(w\qwa c'_j)(a_i\star a'_j)=\sum w\qwa (c_ic'_j)(a_i\star a'_j)=w\qwa (\gamma\star \gamma').
    \end{equation*}
    In the first and the last equalities, we have used the fact that $a_i,a'_j, a_i\star a'_j\in \operatorname{Im}\left(\left.\Psi^{\hbar=0}_{T^*\PPP}\right|_{\scriptstyle{\sAGr_{\hh=0}}}\right)$, and hence they are $\qw_P$-invariant by \Cref{maintext_thm_image_shift=qwinvariants}(1).
\end{proof}

\begin{cor}[=\Cref{cor_image_psi}]\label{maintext_cor_image_psi}
    There is a graded ring isomorphism
    \begin{equation}
        \sAGr_{\hbar=0} \simeq \bigl(\bigl(H_{\GG}^{\bullet}(T^*\BBB)[q_G]\bigr)^{\qw_B},\star \bigr). \tag*{\qed}
    \end{equation}
\end{cor}

\subsection{Spherical subalgebra} 
Recall the tDAHA $\HH_{G,\hbar,\dd}$ from \Cref{def_tDAHA}. Let $\ee':=\frac1{|W|}\sum_{w\in W}L_w\in \HH_{G,\hbar,k}$. The subalgebra $\SH_{G,\hbar,k}:=\ee'\,\HH_{G,\hbar,k}\,\ee'$ of $\HH_{G,\hbar,k}$ is called the spherical subalgebra of $\HH_{G,\hbar,\dd}$.

\begin{thm}[=\Cref{thm_spherical_subalgebra}]\label{maintext_thm_spherical_subalgebra}
    $\sAGr$ is isomorphic to $\SH_{G,\hbar,k-\hbar}$ with parameter shift $\dd\mapsto \dd-\hh$.
\end{thm}

Let $\twistdd$ be the unique ring automorphism of
$H_{\bullet}^{\TT\times\Cxh}(\Fl_G)\simeq H_{\bullet}^{T\times\Cxh}(\Fl_G)[\dd]$ sending
$\dd$ to $\dd-\hbar$ and fixing the subring $H_{\bullet}^{T\times\Cxh}(\Fl_G)$. By identifying $\HH_{G,\hbar,\dd}$ with $\sA$ using \Cref{maintext_cor_Coulomb_branch=tDAHA}, \Cref{maintext_thm_spherical_subalgebra} is equivalent to saying $\twistdd(\ee\sA\ee)\simeq \sAGr$, where $\ee:=\frac 1{|W|}\sum_{w\in W} A_w$.

To prove \Cref{maintext_thm_spherical_subalgebra}, we first establish some facts about the shift operator $\bbS_{t_{\dd}}$ associated with the cocharacter $t_\dd$ corresponding to the inclusion $\Cxd\subseteq\TT$, which is obtained by applying the construction from
\Cref{subsection_shift_operators}, with $G$ replaced by $\GG$. In what follows, we take $X=T^*\BBB$.

We will now work with the $\GG$-equivariant Novikov variables. Note that
\begin{equation}\label{H_2_of_flag_variety}
    H_2^{\GG}(X;\Z)\simeq H^G_2(\BBB;\Z)\oplus H_2^{\Cxd}(\pt;\Z)\simeq H_2^G(\BBB;\Z)\oplus \Z\cdot\beta_k,
\end{equation}
where $\beta_k$ corresponds to the cocharacter $t_k$. Thus, the Novikov variables for $\GG$ acting on $X$ is given by
\begin{equation*}
    \Q[[q_{\GG,X}]]\simeq \Q[[q_G]][[q_k]],
\end{equation*}
where $q_k$ corresponds to $\beta_k$. The shift operator $\bbS_{t_k}$ is then an endomorphism of $H_{\TT\times\Cxh}^\bullet(X)_\loc[[q_{\GG}]]$. It is straightforward to see that $\bbS_{t_k}$ is $W$-equivariant.

\begin{lem}\label{lemma_shift_op_dd_twist}
    For any $\Gamma\in H_{\bullet}^{\TT\times\Cxh}(\Fl_G)$ and $\gamma\in H_{\TT\times\Cxh}^{\bullet}(X)[[q_{\GG}]]$, we have
    \[ \Gamma\cdot\bbS_{t_{\dd}}(\gamma)) = \bbS_{t_{\dd}}(\twistdd(\Gamma)\cdot\gamma)).\]
\end{lem}
\begin{proof}
    We may assume $\Gamma=k^n\Gamma'$ for some $\Gamma'\in H^{T\times\Cxh}_\bullet(\Fl_G)$. Then, we have
    \[ \Gamma\cdot\bbS_{t_{\dd}}(\gamma) = \bbS_{t_{\dd}}((k-\hbar)^n\Gamma'\cdot\gamma)=\bbS_{t_{\dd}}(\twistdd(\Gamma)\cdot\gamma),\]
    where the first equality follows from expressing $\Gamma'$ in terms of fixed point basis and applying \Cref{localizedshiftop} for $\GG.$
\end{proof}

Let $i:\Eff(X)\to H_2^G(X;\Z)$ be the canonical injection. Let $\Q[[q_G^+]]$ be the graded completion of $\Q[i(\Eff(X))]$ along the ideal generated by $i(\Eff(X))\setminus \{0\}$. We define
\[H_{\GG\times\Cxh}^\bullet(X)[[q_G^+]]\coloneq H_{\GG\times\Cxh}^\bullet(X)\widehat{\otimes}_{\Q} \Q[[q_G^+]]\subseteq H^\bullet_{\GG\times\Cxh}(X)[[q_G]].\]

\begin{lem}\label{lemma_shift_op_dd_image}
    We have
    \begin{equation*}
        \bbS_{t_k}\bigl(H_{\GG\times\Cxh}^{\bullet}(X)[[q^+_{G}]]\bigr)
        =
        q_\dd\, e(T^*\BBB)\cup H_{\hG\times\Cxh}^\bullet(X)[[q_G^+]].
    \end{equation*}
\end{lem}
\begin{proof}
    We first prove the inclusion $\subseteq$. Note that the Seidel space $E_{t_\dd}$ corresponding to the cocharacter $t_k$ is the total space of the vector bundle $\oh_{\PP^1}(-1)\boxtimes T^*\BBB$ on $\mathbb{P}^1\times \BBB$. Since $T\BBB$ is globally generated, any genus-zero stable map to $E_{t_{\dd}}$ representing a section class must land in the zero section $\PP^1\times \BBB$. It follows that for any section class $\beta$ in $E_{t_{\dd}}$, the moduli space $\cM_{t_\dd,\beta}(X)$ is complete, and the evaluation morphism $\ev_\infty^\beta:\cM_{t_\dd,\beta}(X)\to X$ lands in $\BBB$. Since $[\BBB]$ is Poincar\'e dual to $e(T^*\BBB)$, we see that
\begin{equation*}
    \bbS_{t_k}\bigl(H_{\GG\times\Cxh}^{\bullet}(X)[[q^+_G]]\bigr)
        \subseteq
        e(T^*\BBB)\cup H_{\hG\times\Cxh}^\bullet(X)[[q_{\GG}]].
\end{equation*}
It remains to check that the relevant Novikov variables must lie in $q_{\dd}\Q[[q_G^+]]$. The map
\begin{equation}\label{lemma_shift_op_dd_image_eq1}
    \Eff(E_{t_{\dd}})\simeq \Eff(\BBB)\oplus \Eff(\PP^1)\to H^{\GG}_2(X;\Z)\simeq H_2^G(\BBB;\Z)\oplus \Z\cdot \beta_k
\end{equation}
is given by
\begin{equation*}
    (\beta',n[\PP^1])\mapsto (i(\beta'), n\beta_k).
\end{equation*}
Hence, $q^{\overline{\beta}}\in q_\dd \Q[[q_G^+]]$ for any $\beta\in \Eff(E_{t_{\dd}})$ whose projection to $\Eff(\PP^1)$ is $[\PP^1]$. The inclusion follows.

Let $\mathfrak{m}\subseteq \Q[[q_G^+]]$ be the kernel of the homomorphism $\Q[[q_G^+]]\to \Q$ sending $g$ to $g(0)\in \Q$. To complete the proof, it suffices to prove the equality modulo $\mathfrak{m}$. The only curve class in the decomposition $\Eff(E_{t_{\dd}})\simeq \Eff(\BBB)\oplus \Eff(\PP^1)$ that contributes to $\bbS_{t_\dd}$ modulo $\mathfrak{m}$ is $\beta_0=[\PP^1]$. We have a natural identification $\cM_{t_\dd,\beta_0}(X)\simeq \BBB$, with
$[\cM_{t_\dd,\beta_0}(X)]^{\vir}=[\BBB]$, and the evaluation morphisms
$\ev_0^{\beta_0}$ and $\ev_\infty^{\beta_0}$ are both the canonical inclusions
$\BBB\hookrightarrow X$. Hence,
\begin{equation}\label{eq_shift_op_dd_image_eq1}
    \bbS_{t_\dd}(\gamma)
    =
    q_\dd e(T^*\BBB)\cup \Phi_{t_k}(\gamma)
    \pmod{\mathfrak{m}}
\end{equation}
for any $\gamma\in H_{\hG\times\Cxh}^\bullet(X)$. This proves the lemma.
\end{proof}

\begin{cor}\label{lemma_shift_op_dd_1}
    We have $\bbS_{t_\dd}(1)=q_\dd g\, e(T^*\BBB)$ for some $g\in \Q[[q_G^+]]$ with $g(0)=1$.
\end{cor}
\begin{proof}
    By \Cref{lemma_shift_op_dd_image} and the definition of $\bbS_{t_k}$, we have $\bbS_{t_\dd}(1)=q_\dd e(T^*\BBB)\cup g$ for some $g\in H_{\hG\times\Cxh}^\bullet(X)[[q_G^+]]$. Note that $\deg \bbS_{t_\dd}(1)=0$, but $\deg q_\dd=-\dim \BBB$. Therefore, we must have $\deg g=0$. In other words, $g\in \Q[[q_G^+]]$. The claim $g(0)=1$ follows from \Cref{eq_shift_op_dd_image_eq1}. 
\end{proof}

We will need two more lemmas for the proof of \Cref{maintext_thm_spherical_subalgebra}.
\begin{lem}\label{lemma_euler_class}
    $e(T^*\BBB)=(-1)^{\dim\BBB}\sum_{u\in W}\stabm(u)$.
\end{lem}
\begin{proof}
    Both sides are $W$-invariant, so it suffices to compare their restrictions
    to $eB$, which is obvious.
\end{proof}

\begin{lem}\label{spherical_nonlocalized}
    $\twistdd(\ee\sA\ee)\cdot 1 \subseteq \bigl(H_{\GG\times\Cxh}^\bullet(T^*\BBB)[q_G]\bigr)^{\qw_B}$.
\end{lem}
\begin{proof}
    We already know from \Cref{Namikawa_commutes_A} and \Cref{imPsi_is_qw_invariant} that $\twistdd(\ee\sA\ee)\cdot 1$ is $\qw_B$-invariant. Hence, it suffices to prove $\twistdd(\ee\sA\ee)\cdot 1\subseteq H_{\GG\times\Cxh}^\bullet(T^*\BBB)[[q_G]]$. By \Cref{lemma_shift_op_dd_twist}, we have $\twistdd(\ee\sA\ee)\cdot 1=\bbS_{t_{\dd}}^{-1}\left(\ee\sA\ee\cdot\bbS_{t_{\dd}}(1)\right)$. By \Cref{lemma_shift_op_dd_1} and \Cref{lemma_euler_class}, we have 
    \[\bbS_{t_{\dd}}(1)=q_{\dd}ge(T^*\BBB)=(-1)^{\dim\BBB}q_{\dd}g\sum_{u\in W}\stabm(u)\]
    for some $g\in\Q[[q_G^+]]$ with $g(0)=1$. Hence, by \Cref{lemma_shift_op_dd_image}, it suffices to prove 
    \[ \ee\sA\ee\cdot\left(\sum_{u\in W}\stabm(u)\right)\subseteq e(T^*\BBB)\cup H_{\GG\times\Cxh}^{\bullet}(T^*\BBB)[q_G].\]

    By \Cref{Cohomology_determined_by_restriction} and the fact that $\ee$ is left $W$-invariant, it remains to show that for any $a\in \ee\sA\ee$,
    \begin{equation}\label{eq_spherical_nonlocalized}
        \left.\left(a\cdot\Bigl(\sum_{u\in W}\stabm(u)\Bigr)\right)\right|_{eB}\in e(T^*_{eB}\BBB)\cdot \Q[\lt][\hh,\dd][q_G].
    \end{equation}
    Indeed, \eqref{eq_spherical_nonlocalized} holds for any $a\in\sA$. To see this, we may assume $a=f\DD_\xx$ for $f\in\Q[\lt][\hh,\dd]$ and $\xx=wt_{\lambda}\in\Waffe$. Then
    \begin{equation*}
        \left.\left(a\cdot\Bigl(\sum_{u\in W}\stabm(u)\Bigr)\right)\right|_{eB} 
        =\sum_{u\in W}fq^{u^{-1}(\lambda)}\left.\stabm(wu)\right|_{eB}= (-1)^{\dim \BBB}e(T^*_{eB}\BBB)fq^{w(\lambda)}.   \qedhere
    \end{equation*} 
\end{proof}

\begin{proof}[Proof of \Cref{maintext_thm_spherical_subalgebra}]
    Define 
\begin{align*}
    \psi:\twistdd(\ee\sA\ee)&\to H_{\bullet}^{\TT\times\Cxh}(\Gr_G) \\
    a&\mapsto a\cdot_\Gr 1
\end{align*}
    Observe that $\psi$ actually lands in $H_{\bullet}^{\GG\times\Cxh}(\Gr_G)$, and hence it is a ring homomorphism by \Cref{psi_is_ringhomo}. 
    
   It remains to check that $\psi$ is injective with image $\sAGr$. We claim that it suffices to show $\operatorname{Im}(\psi)\subseteq \sAGr$. Indeed, by a standard argument using induction on degree, the result follows from the analogous claim for $\hh=\dd=0$, which is easy: $\twistdd(\ee\sA\ee)|_{\hbar=k=0}=\bigl(\sum_{\mu\in\Lambda}\Q[\lt]\cdot[t_{\mu}]\ee_0\bigr)^W$ for $\ee_0:=\frac{1}{|W|}\sum_{w\in W}[w]$ and $\sAGr_{\hbar=k=0}=\bigl(\sum_{\mu\in\Lambda}\Q[\lt]\cdot[t_{\mu}]\bigr)^W$ (see \cite[6(vi)]{BFN}), and $\psi$ becomes the restriction of the map that sends $f[t_{\mu}]\ee_0$ to $f[t_{\mu}]$ for $f\in \Q[\lt]$ and $\mu\in \Lambda$. 

    Recall from \Cref{maintext_thm_image_shift=qwinvariants} that $\Psi_{T^*\BBB}$ is injective, so it suffices to show that $\Psi_{T^*\BBB}(\operatorname{Im}(\psi))=\Psi_{T^*\BBB}(\sAGr)$. By \Cref{maintext_thm_image_shift=qwinvariants}, we have $\Psi_{T^*\BBB}(\sAGr)=\bigl(H_{\GG\times\Cxh}^\bullet(T^*\BBB)[q_G]\bigr)^{\qw_B}$. The desired equality then follows from \Cref{Results_of_CCL}(2) and \Cref{spherical_nonlocalized}. This completes the proof.
    \end{proof}
    
\printbibliography
\end{document}